\documentclass[11pt]{amsart}

\usepackage{times}
\usepackage{tikz}
\usepackage{a4wide}
\usepackage{amssymb}
\usepackage{amsmath}
\usepackage{epsfig}
\usepackage[sans]{dsfont}

\newcommand{\one}{\mathds{1}}
\newcommand{\Sbf}{{\bf S}}
\newcommand{\Dbf}{{\bf D}}
\newcommand{\Ec}{{\mathcal E}}
\newcommand{\M}{{\mathcal M}}
\newcommand{\Lcal}{{\mathcal L}}
\newcommand{\Sc}{{\mathcal S}}
\newcommand{\Dc}{{\mathcal D}}

\newcommand{\R}{{\mathbb R}}

\newcommand{\C}{{\mathbb C}}
\newcommand{\N}{{\mathbb N}}

\newcommand{\im}{\operatorname{Im}}

\renewcommand{\im}{\operatorname{Im}}

\theoremstyle{plain}

\newtheorem{theorem}{Theorem}
\newtheorem{corollary}[theorem]{Corollary}
\newtheorem{lemma}[theorem]{Lemma}
\newtheorem{proposition}[theorem]{Proposition}
\newtheorem{remark}[theorem]{Remark}

\newtheorem{definition}[theorem]{Definition}

\theoremstyle{definition}

\numberwithin{theorem}{section}
\numberwithin{equation}{section}

\renewcommand{\Im}{{\rm Im}}
\renewcommand{\Re}{{\rm Re}}

\title[Magnetic resonances for exterior problems]{Counting function of magnetic resonances for exterior problems}

\author[V. Bruneau]{Vincent Bruneau}
\address{Institut de Math\'ematiques de Bordeaux, UMR 5251 du CNRS,
Universit\'e de Bordeaux, 351 cours de la Lib\'eration, 33405 Talence cedex, France}
\email{vbruneau@math.u-bordeaux1.fr}
\author[D. Sambou]{Diomba Sambou}
\address{Departamento de Matem\'aticas, Facultad de Matem\'aticas, Pontificia Universidad
Cat\'olica de Chile, Vicu\~na Mackenna 4860, Santiago de Chile}
\email{disambou@mat.uc.cl}


\def\phi {\varphi}

\newcommand{\dbdN}{\partial_N^A}
\newcommand{\dbdR}{\partial_\Sigma}

\newcommand{\om}{\Omega}

\newcommand{\Lc}{{\mathcal L}}

\newcommand{\xp}{x_\perp}

\newcommand{\bel}{\begin{equation} \label}
\newcommand{\ee}{\end{equation}}

\def\l{{\Lambda}}

\def\<{\langle}
\def\>{\rangle}

\usepackage{color}

\definecolor{gr}{rgb}   {0.,   0.69,   0.23 }
\definecolor{bl}{rgb}   {0.,   0.5,   1. }
\definecolor{mg}{rgb}   {0.85,  0.,    0.85}
\definecolor{yl}{rgb}   {0.8,  0.7,   0.}

\newcommand{\Bk}{\color{black}}
\newcommand{\Rd}{\color{red}}

\begin{document}

\begin{abstract}
We study the asymptotic distribution of the resonances  near 
the Landau levels $\Lambda_q =(2q+1)b$, $q \in \mathbb{N}$, of 
the Dirichlet (resp. Neumann, resp. Robin) realization in the 
exterior of a compact domain of $\mathbb{R}^3$ of the 3D 
Schr\"{o}dinger Schr\"odinger operator with constant magnetic field of scalar 
intensity $b>0$. We investigate the corresponding resonance 
counting function and obtain the main asymptotic term. 
In particular, we prove the accumulation of resonances at the 
Landau levels and the existence of resonance free sectors. In 
some cases, it provides the discreteness of the set of embedded 
eigenvalues near the Landau levels.
\end{abstract}

\maketitle

\section{Introduction}\label{S:intro}

\quad It is now well known that perturbations of magnetic Schr\"odinger operators can generate spectral accumulations near the Landau levels. In the $2D$ case, the free Hamiltonian (the {\it Landau Hamiltonian}) admits pure point spectrum with eigenvalues (the so called {\it Landau levels}) of infinite multiplicity. Its perturbations by an electric potential of definite sign (even if it is compactly supported)  produce concentration of eigenvalues at the Landau levels (see \cite{Ra90_01}, \cite{RaWa02_01}, \cite{MeRo03},   \cite{FiPu06}). 
More recently, similar phenomena are obtained for perturbations by obstacle (see \cite{PuRo07} for the Dirichlet problem, \cite{Pe09} for the Neumann problem and \cite{GoKaPe13} for the Robin boundary condition). Let us  also mention the  work \cite{KloRai09} where is considered potential perturbations which are not of fixed sign.

The study of the $3D$ Schr\"odinger operator is more complicated because the spectrum of the free Hamiltonian is continuous (it is $[b, + \infty)$ where $b>0$ is the strength of the constant magnetic field). For perturbations of such operators, the spectral concentration can be analyse on several way. For example, it is possible to prove that some axisymmetric perturbations can produce an infinite number of embedded eigenvalues near the Landau levels  (see \cite{BoBrRa14_02}). In a more general framework it is stated that the Landau levels are singularities of the Spectral Shift Function  (see \cite{FeRa04_01}) and are accumulation points of resonances (see \cite{BoBrRa07_01}, \cite{BoBrRa14_01}). These results are done for a wide class of potentials of definite sign, but it is also important to consider the cases of obstacle perturbations. For example, magnetic boundary problems appear in the Ginzgurg-Landau theory of superconductors, in the theory of Bose-Einstein condensat
 es, 
and in the study of edge states in Quantum Mechanics (see for instance \cite{DBiPul99}, \cite{HoSmi02}, \cite{AfHe09}, \cite{FoHe10},...).

In this paper, we consider the 3D
Schr\"odinger operators with constant magnetic field of strength
$b>0$, pointing at the $x_3$-direction. For the magnetic potential $A=(-b\frac{x_2}2, b\frac{x_1}2, 0)$ it is given by :
\begin{equation} \label{gdr0}
\begin{split}
H_0: = - \Big(  \nabla^A \Big)^2  =  \Big( D_{1} + & \frac{b}{2} x_2 \Big)^2 + \Big( D_{2} - \frac{b}{2} x_1 \Big)^2  + D_{3}^2  , \qquad D_j := -i \frac{\partial}{\partial x_j}, \\
&  \nabla_j^A := \nabla_{x_j}-iA_j.
\end{split}
\end{equation}
Set $\xp := (x_1, x_2) \in \R^2$. Using the
representation $L^2(\R^3) = L^2(\R_{\xp}^2) \otimes
L^2(\R_{x_3})$, $H_0$ admits the decomposition
\begin{equation} \label{gdr10}
H_0 = H_{\rm Landau} \otimes I_3 + I_{\perp} \otimes
\Big( - \frac{\partial^2}{\partial x_{3}^2} \Big)
\end{equation}
with
\begin{equation} \label{gdr4}
 H_{\text{Landau}}: = \Big( D_1 + \frac{b}{2} x_2 \Big)^2 + \Big( D_2 - \frac{b}{2} x_1 \Big)^2 ,
\end{equation}
the Landau Hamiltonian
and $I_3$ and $I_\perp$ being the identity operators in
$L^2(\R_{x_3})$ and $L^2(\R_{\xp}^2)$ respectively. The spectrum of $H_{\text{Landau}}$ consists of the
so-called Landau levels $\Lambda_q=(2q+1)b$,  $q \in \N : = \{0,1,2, \ldots\}$,
and ${\rm dim}\,{\rm Ker}(H_{\text{Landau}} - \Lambda_q) = \infty$. Consequently,
\begin{equation*}
\sigma ( H_0 ) = \sigma_{\rm ac} ( H_0 ) = [ b , + \infty [ ,
\end{equation*}
and the Landau levels play the role of thresholds
in the spectrum of $H_0$. 

Let us introduce the obstacle perturbation. Let $K\subset \R^3$ be a  compact domain with smooth boundary $\Sigma$ and let $\Omega:= \R^3 \setminus K$. We denote by $\nu$ the unit outward  normal vector of the boundary $\Sigma$ and by $\dbdN:= \nabla^A \cdot \nu$ the magnetic normal derivative. For $\gamma$ a smooth real valued function on $\Gamma$, 
we introduce the following operator on $\Sigma$:
$$\dbdR^{A,\gamma} := \nabla^A \cdot \nu + \gamma. $$
From now, $\gamma$ is fixed and  if it does not lead to confusion, we shall omit the index ${A,\gamma}$ and write $\dbdR$ for $\dbdR^{A,\gamma}$.


In the following lines let us define $H_\Omega^\gamma$ (resp. $H_\Omega^\infty$) the Neumann and Robin (resp. Dirichlet) realization of $- \Big(  \nabla^A \Big)^2$ on $\Omega$.

\textbf{Neumann and Robin realizations of $- \Big(  \nabla^A \Big)^2$:}
\\\\
The operator $H_\Omega^\gamma$ is defined by
\begin{equation}\label{dom1}
\begin{cases} 
H_\Omega^\gamma u = - \Big( \nabla^A \Big)^2 u, \quad u \in Dom \big( H_\Omega^\gamma \big), \\ 
Dom \big( H_\Omega^\gamma \big) := \Big\lbrace u \in L^{2}(\Omega) : \Big( \nabla^A \Big)^j u \in L^{2}(\Omega), \hspace*{0.1cm} j = 1, 2 \hspace*{0.1cm} : \hspace*{0.1cm} \dbdR^{A,\gamma} u = 0 \hspace*{0.2cm} \text{on} \hspace*{0.2cm} \Sigma \Big\rbrace. 
\end{cases}
\end{equation}
Actually, $H_\Omega^\gamma$ is the self-adjoint operator associated to the closure of the quadratic form
\begin{equation}\label{q1}
Q_{\Omega}^{\gamma}(u) = \int_{\Omega} \big\vert \nabla^A u \big\vert^{2} dx + \int_{\Sigma} \gamma \vert u \vert^{2} d\sigma, \qquad x := (x_{\perp},x_{3}),
\end{equation}
originally defined in the magnetic Sobolev space $H_{A}^{1}(\Omega) := \big\lbrace u \in L^{2}(\Omega) : \nabla^A u \in L^{2}(\Omega) \big\rbrace$.

The Neumann realization corresponds to $\gamma=0$.


\textbf{Dirichlet realization of $- \Big(  \nabla^A \Big)^2$:}
\\\\
The operator $H_\Omega^\infty$ is defined by
\begin{equation}\label{dom2}
\begin{cases} 
H_\Omega^\infty u = - \Big( \nabla^A \Big)^2 u, \quad u \in Dom \big( H_\Omega^\infty \big), \\ 
Dom \big( H_\Omega^\infty \big) := \Big\lbrace u \in L^{2}(\Omega) : \Big( \nabla^A \Big)^j u \in L^{2}(\Omega), \hspace*{0.1cm} j = 1, 2 \hspace*{0.1cm} : \hspace*{0.1cm} u = 0 \hspace*{0.2cm} \text{on} \hspace*{0.2cm} \Sigma \Big\rbrace. 
\end{cases}
\end{equation}
Actually, $H_\Omega^\infty$ is the self-adjoint operator associated to the closure of the quadratic form
\begin{equation}\label{q3}
Q_{\Omega}^{\infty}(u) = \int_{\Omega} \big\vert \nabla^A u \big\vert^{2} dx, \qquad x := (x_{\perp},x_{3}),
\end{equation}
originally defined on $C_{0}^{\infty}(\Omega)$. 


\begin{remark}\label{r1}

 The magnetic Schr\"odinger operator $H_{0}$ defined by \eqref{gdr0} is the self-adjoint operator associated to the closure of the quadratic form \eqref{q3} with $\Omega=\R^3$.
\end{remark}

As compactly supported perturbations of the elliptic operator $H_0$, the operators $H_\Omega^\infty$ and  $H_\Omega^\gamma$ are relatively compact perturbations of $H_0$ and we have:

\begin{proposition}
For $l=\infty, \gamma$, the essential spectrum of $H_\Omega^l$ coincide with these of $H_0$:
$$\sigma_{ess}(H_\Omega^\infty)=\sigma_{ess}(H_\Omega^\gamma)=\sigma_{ess}(H_0)=\sigma(H_0)= [b, + \infty).$$
\end{proposition}

This result is proved in a more general context in \cite{KaPe13}. It is also a consequence of some resolvent equations as in Section \ref{s3}. 

In order to define the resonances, let us recall analytic properties of the free resolvent. 
Let $ {\mathcal M}$ be the connected infinite-sheeted covering of $\C\setminus \cup_{q\in \N}\{\Lambda_q\}$ where each function $z \mapsto \sqrt{z-\Lambda_q}$, $q \in \N$ is analytic. Near a Landau level $\Lambda_q$ this Riemann surface $ {\mathcal M}$ can be parametrized by $z_q(k)= \Lambda_q +k^2$, $k \in \C^*$, $|k| \ll 1$ (for more details, see Section 2 of \cite{BoBrRa07_01}). For $\epsilon>0$, we denote by ${\mathcal M}_{\epsilon}$ the set of the points
$z \in {\mathcal M}$ such that for each $q \in \N$, we have $\Im \sqrt{z-\Lambda_q} > -\epsilon$. We have $\cup_{\epsilon>0} {\mathcal M}_{\epsilon} = {\mathcal M}$.

\begin{proposition}{\cite[Proposition 1]{BoBrRa07_01}}\label{propMeroR0}

For each $\epsilon >0$, the operator 
$$R_0(z)=(H_0-z)^{-1} : e^{-\epsilon \langle x_{3} \rangle} L^2 (\R^{3} ) \rightarrow e^{\epsilon \langle x_{3} \rangle} L^2 (\R^{3})$$
 has a holomorphic extension (still denoted by $R_0(z)$) from the open upper half-plane ${\mathbb C}_+
: = \{z \in {\mathbb C}; \ {\rm Im}\,z > 0\}$ to  ${\mathcal M}_{\epsilon}$. 
\end{proposition}



Since $H_\Omega^\infty$ and  $H_\Omega^\gamma$ are compactly supported perturbations of $H_0$, using some resolvent equations and the analytic Fredholm theorem, from Proposition \ref{propMeroR0}, we deduce  meromorphic extension of the resolvents of $H_\Omega^\infty$ and  $H_\Omega^\gamma$. It can be done following   the "black box" framework developed for perturbation of the Laplacian (as in \cite{Vod92}, \cite{SjoZwo91})  or by introducing auxiliary operators as in Section \ref{s3} (see Corollary \ref{Cor:PoleReso}). Then we are able to define the resonances:

\begin{definition}
For $l=\infty, \gamma$, we define the resonances for $H_\Omega^l$  as the poles of the meromorphic extension of the resolvent 
$$(H_\Omega^l-z)^{-1}:  e^{-\epsilon \langle x_{3} \rangle} L^2 (\Omega) \rightarrow e^{\epsilon \langle x_{3} \rangle} L^2 (\Omega).$$
These poles (i.e. the resonances) and the rank of their residues (the multiplicity of the resonance) do not depend on $\epsilon > 0 $.

\end{definition}

Our goal is to study the distribution of the resonances of $H_\om^\infty$ and $H_\om^\gamma$ near the Landau levels.
We will essentially prove that the distribution of the resonances of $H_\om^\infty$ (resp. $H_\om^\gamma$) near the Landau levels is essentially governed by the distribution of resonances of $H_0+\one_{K}$ (resp. $H_0-\one_{K}$) which is known thanks to 
\cite{BoBrRa14_01} .

The article is organized as follows. Our main results and their corollaries are formulated and discussed in Section \ref{s2}. In Section \ref{s3}, we show how we can reduce the study of the operators $H_\om^\infty$ and $H_\om^\gamma$ near the Landau levels to some compact perturbations, of fixed sign, of $H_0^{-1}$. By this way, in Subsection \ref{ss3,3}, we bring out the relation between the perturbed operators of $H_0^{-1}$ and the Dirichlet-Neumann and the Neumann-Dirichlet operators. Section \ref{S:proof} is devoted to the proofs of our main results. In Sections \ref{S:sec5} and \ref{S:sec6}, exploiting the fact that the Dirichlet-Neumann and the Neumann-Dirichlet operators are elliptic pseudo differential operators on the boundary $\Sigma$, we show how we can reduce the analyse of the perturbed operators to that of Toeplitz operators with symbol supported near the obstacle. Section \ref{s7} is devoted to the computational proof of the lemma needed to prove that the Diric
 hlet-Neumann and the Neumann-Dirichlet operators are elliptic pseudo differential operators on the boundary $\Sigma$.

\section{Formulation of the main results}\label{s2}

\quad  For $l = \infty, \gamma$, let $H_\om^l$ be the magnetic Schr\"odinger operators defined by \eqref{dom1} and \eqref{dom2} and let us  denote by $\text{Res}\big(H_\om^l\big)$ the corresponding resonances sets. 

Near a Landau level $\Lambda_q$, $q \in \N$, we parametrize the resonances $z_q$ by $z_q(k) = \Lambda_q + k^2$ with $\vert k \vert <\!\!<1$.

Our main result gives the localization of the resonances of  $H_\om^\infty$ and $H_\om^\gamma$ near the Landau levels $\Lambda_{q}$, $q \in \N$, together with an asymptotic expansion of the resonances counting function in small annulus adjoining  $\Lambda_{q}$, $q \in \N$. As consequences we obtain some informations concerning eigenvalues.


\begin{theorem}\label{theo1}
Let $K \subset \R^3$ be a  smooth compact domain. Fix a Landau level $\Lambda_{q}$, $q \in \N$, such that $K$ does not produce an isolated resonance at $\Lambda_q$ (see Definition \ref{defnreso}).
Then the resonances $z_{q}(k) = \Lambda_{q} + k^{2}$ of $H_\om^\infty$ and $H_\om^\gamma$, with $\vert k \vert <\!\!<1$ sufficiently small, satisfy:  

$\textup{\textbf{(i)}}$ For $0 < r_{0} < \sqrt{2b}$ fixed and $l = \infty, \gamma$
\begin{equation*}
\displaystyle \sum_{\substack{z_{q}(k) \hspace{0.5mm} \in \hspace{0.5mm} \textup{Res}(H_\om^l) \\ r < \vert k \vert < r_{0}}} \textup{mult} \big( z_{q}(k) \big) \sim \frac{\vert \ln r \vert}{\ln \vert \ln r \vert} \big( 1 + o(1) \big), \qquad r \searrow 0.
\end{equation*}

$\textup{\textbf{(ii)}}$ For the Dirichlet exterior problem ($l=\infty$), the resonances $z_{q}$ are far from the real axis in the sense that there exists $r_0>0$ such that  $k=\sqrt{z_q-\Lambda_q}$, $|k | < r_0$ satisfies:
\begin{equation*}
\Im(k) \leq 0, \qquad \Re(k) = o(\vert k \vert).
\end{equation*}

$\textup{\textbf{(iii)}}$ For the Neumann-Robin exterior problem ($l=\gamma$), the resonances $z_{q}$ are close to the real axis, below $\Lambda_q$, in the sense that there exists $r_0>0$ such that  $k=\sqrt{z_q-\Lambda_q}$, $|k | < r_0$ satisfies:

\begin{equation*}
\Im(k) \geq 0, \qquad \Re(k) = o(\vert k \vert).
\end{equation*}

\end{theorem}

In particular, near the first Landau level $\Lambda_0=b$, using that the only poles $z_{0}(k) = \Lambda_{0} + k^{2}$, with $\Im k >0$, are the eigenvalues below $\Lambda_0$, and the fact that the Dirichlet operator is a non negative perturbation of $H_0$ (see Lemma \ref{lemVM}), we have :

\begin{corollary}\label{theo2}

$\textup{\textbf{(i)}}$ The Robin (resp. Neumann) exterior operator $H_\om^\gamma$ (resp. $H_\om^0$) has an increasing sequence of eigenvalues $\{\mu_j\}_j$ which accumulate at  $\Lambda_0$ with the distribution:
\begin{equation*}
\displaystyle \# \{ \mu_j \in \sigma_p(H_\om^\gamma) \cap (-\infty , \Lambda_0-\lambda) \}
 \sim \frac{\vert \ln \lambda \vert}{2\ln \vert \ln \lambda \vert} \big( 1 + o(1) \big), \qquad \lambda \searrow 0,
\end{equation*}

$\textup{\textbf{(ii)}}$ The Dirichlet exterior operator $H_\om^\infty$ has no eigenvalues below $\Lambda_0$.

\end{corollary}

Moreover, since the embedded eigenvalues of the operator $H_\om^l$ in $[b, + \infty) \setminus \cup_{q=0}^{\infty} \lbrace \Lambda_{q} \rbrace$ are the resonances $z_{q}(k)$ with $k \in e^{i\lbrace 0,\frac{\pi}{2} \rbrace} ]0,\sqrt{2b}[$, then an immediate consequence of Theorem \ref{theo1} $\textup{\textbf{(ii)}}$ and $\textup{\textbf{(iii)}}$ is the absence of embedded eigenvalues of $H_\om^\infty$ in $]\Lambda_{q} - r_{0}^{2},\Lambda_{q}[ \cup ]\Lambda_{q},\Lambda_{q} + r_{0}^{2}[$ and of embedded eigenvalues of $H_\om^\gamma$ in $ ]\Lambda_{q},\Lambda_{q} + r_{0}^{2}[$, for $r_0$ sufficiently small. Hence we have the following result:

\begin{corollary}
In $[b, + \infty)  \setminus\cup_{q=0}^{\infty} \big\lbrace \Lambda_{q} \big\rbrace$ (resp. in $[b, + \infty) \setminus \cup_{q=1}^{\infty} \big\lbrace ]\Lambda_{q} - r_{0}^{2},\Lambda_{q}[ \big\rbrace$) the  embedded eigenvalues of the operator $H_\om^\infty$  (resp. $H_\om^\gamma$), form a discrete set.
\end{corollary}

\begin{figure}[h]\label{fig 1}
\begin{center}
\tikzstyle{+grisEncadre}=[fill=gray!60]
\tikzstyle{blancEncadre}=[fill=white!100]
\tikzstyle{grisEncadre}=[fill=gray!20]
\begin{tikzpicture}[scale=1]

\begin{scope}
\draw [grisEncadre] (0,0) -- (90:2) arc (90:360:2);
\draw [blancEncadre] (0,0) -- (0:2) arc (0:90:2);
\draw [+grisEncadre] (0,0) -- (-113:2) arc (-113:-67:2);

\draw [<->] (0,-2.3) arc (-90:-113:2.3);
\draw [<->] (0,-2.3) arc (-90:-67:2.3);

\draw (0,0) -- (-1.05,-2.5);

\draw (0,0) -- (1.05,-2.5);
\draw (0,0) -- (1.05,-2.5);

\draw (0,0) circle(0.5);
\draw [->] [thick] (-2.5,0) -- (2.3,0);
\draw (2.3,0) node[right] {\tiny{$\Re(k)$}};
\draw [->] [thick] (0,-2.5) -- (0,2.7);
\draw (0,2.5) node[right] {\tiny{$\Im(k)$}};
\draw (-0.15,0.02) node[above] {\tiny{$r$}};
\draw (0,0) -- (-0.45,0.17);
\draw (0,0) -- (1.73,1);
\draw (1,0.5) node[above] {\tiny{$r_{0}$}};

\draw (-0.2,-2.1) node[above] {\tiny{$S_{\theta}$}};
\draw (-0.4,-2.35) node[above] {\tiny{$\theta$}};
\draw (0.4,-2.35) node[above] {\tiny{$\theta$}};

\node at (-1.7,-0.4) {\tiny{$\times$}};
\node at (-1.3,-0.2) {\tiny{$\times$}};
\node at (-1.4,-1.2) {\tiny{$\times$}};
\node at (-0.8,-1) {\tiny{$\times$}};
\node at (-0.6,-0.5) {\tiny{$\times$}};

\node at (-1.3,0.2) {\tiny{$\times$}};
\node at (-0.8,1) {\tiny{$\times$}};

\node at (0.9,-1.4) {\tiny{$\times$}};
\node at (1.8,-0.2) {\tiny{$\times$}};
\node at (1.4,-1.2) {\tiny{$\times$}};
\node at (0.7,-0.8) {\tiny{$\times$}};
\node at (0.8,-0.3) {\tiny{$\times$}};


\node at (-0.02,-0.9) {\tiny{$\times$}};
\node at (-0.02,-0.7) {\tiny{$\times$}};
\node at (-0.02,-0.5) {\tiny{$\times$}};
\node at (-0.02,-0.3) {\tiny{$\times$}};
\node at (-0.02,-0.2) {\tiny{$\times$}};

\node at (-0.12,-0.9) {\tiny{$\times$}};
\node at (-0.12,-0.7) {\tiny{$\times$}};
\node at (-0.12,-0.5) {\tiny{$\times$}};
\node at (-0.12,-0.3) {\tiny{$\times$}};

\node at (-0.22,-0.9) {\tiny{$\times$}};
\node at (-0.22,-0.7) {\tiny{$\times$}};

\node at (-0.32,-0.9) {\tiny{$\times$}};

\node at (-0.5,-1.7) {\tiny{$\times$}};
\node at (-0.3,-1.3) {\tiny{$\times$}};

\node at (-0.2,-1.5) {\tiny{$\times$}};
\node at (-0.1,-1.3) {\tiny{$\times$}};

\node at (-0.2,-1.1) {\tiny{$\times$}};


\node at (0.12,-0.9) {\tiny{$\times$}};
\node at (0.12,-0.7) {\tiny{$\times$}};
\node at (0.12,-0.5) {\tiny{$\times$}};
\node at (0.12,-0.3) {\tiny{$\times$}};

\node at (0.22,-0.9) {\tiny{$\times$}};
\node at (0.22,-0.7) {\tiny{$\times$}};

\node at (0.32,-0.9) {\tiny{$\times$}};

\node at (0.5,-1.7) {\tiny{$\times$}};
\node at (0.3,-1.3) {\tiny{$\times$}};

\node at (0.2,-1.5) {\tiny{$\times$}};
\node at (0.1,-1.3) {\tiny{$\times$}};

\node at (0.2,-1.1) {\tiny{$\times$}};

\node at (2.2,1.8) {\tiny{$\textup{Dirichlet}$}};
\end{scope}

\begin{scope}[xshift=6cm]
\draw [grisEncadre] (0,0) -- (90:2) arc (90:360:2) -- cycle;
\draw [blancEncadre] (0,0) -- (0:2) arc (0:90:2) -- cycle;
\draw [+grisEncadre] (0,0) -- (90:2) arc (90:113:2) -- cycle;
\draw (0,0) -- (90:2) arc (90:113:2) -- cycle;

\draw [<->] (0,2.3) arc (90:113:2.3);

\draw (0,0) circle(0.5);
\draw [->] [thick] (-2.5,0) -- (2.3,0);
\draw (2.3,0) node[right] {\tiny{$\Re(k)$}};
\draw [->] [thick] (0,-2.5) -- (0,2.7);
\draw (0,2.7) node[right] {\tiny{$\Im(k)$}};
\draw (0.26,-0.40) node[above] {\tiny{$r$}};
\draw (0,0) -- (0.24,-0.45);
\draw (0,0) -- (1.73,1);
\draw (0,0) -- (-1.05,2.5);
\draw (1,0.5) node[above] {\tiny{$r_{0}$}};
\draw (-0.4,1.88) node[above] {\tiny{$\theta$}};

\node at (-1.4,-1.2) {\tiny{$\times$}};
\node at (-0.6,-0.5) {\tiny{$\times$}};

\node at (0.9,-1.4) {\tiny{$\times$}};
\node at (1.8,-0.2) {\tiny{$\times$}};

\node at (-0.02,0.9) {\tiny{$\times$}};
\node at (-0.02,0.7) {\tiny{$\times$}};
\node at (-0.02,0.5) {\tiny{$\times$}};
\node at (-0.02,0.3) {\tiny{$\times$}};
\node at (-0.02,0.2) {\tiny{$\times$}};

\node at (-0.12,0.9) {\tiny{$\times$}};
\node at (-0.12,0.7) {\tiny{$\times$}};
\node at (-0.12,0.5) {\tiny{$\times$}};
\node at (-0.12,0.3) {\tiny{$\times$}};

\node at (-0.22,0.9) {\tiny{$\times$}};
\node at (-0.22,0.7) {\tiny{$\times$}};

\node at (-0.32,0.9) {\tiny{$\times$}};

\node at (-1.7,0.4) {\tiny{$\times$}};
\node at (-1.3,0.2) {\tiny{$\times$}};
\node at (-1.4,1.2) {\tiny{$\times$}};
\node at (-0.8,1) {\tiny{$\times$}};
\node at (-0.6,0.5) {\tiny{$\times$}};

\node at (-0.5,1.7) {\tiny{$\times$}};
\node at (-0.3,1.3) {\tiny{$\times$}};

\node at (-0.2,1.7) {\tiny{$\times$}};
\node at (-0.1,1.3) {\tiny{$\times$}};

\node at (-0.2,1.1) {\tiny{$\times$}};

\node at (2.9,1.8) {\tiny{$\textup{Neumann-Robin}$}};
\draw (-0.3,1.5) node {\tiny{$S_{\theta}$}};
\end{scope}
\end{tikzpicture}
\caption{\textup{\textbf{Localisation of the resonances in variable $k$:} For $r_{0}$ sufficiently small, the resonances $z_{q}(k) = \Lambda_{q} + k^{2}$ of the operators $H_{\Omega}^{l}$, $l = \infty, \gamma$ near a Landau level $\Lambda_{q}$, $q \in \N$, are concentrated in the sectors $S_{\theta}$. For $l = \infty$ they are concentrated near the semi-axis $-i]0,+\infty)$ in both sides, while they are concentrated near the semi-axis $ i]0,+\infty)$ on the left for $l = \gamma$.}}
\end{center}
\end{figure}

To our best knowledge, the above results are new even concerning the discret spectrum. 
However, they are not surprising. Similar results hold for perturbations by potentials (see \cite{Ra90_01} for eigenvalues, et \cite{BoBrRa07_01}, \cite{BoBrRa14_01} for resonances) and for exterior problems in the 2D case concerning accumulation of eigenvalues at the Landau levels  (see \cite{PuRo07}, \cite{Pe09}, \cite{GoKaPe13}). In comparison with previous works, the spectral study of obstacle perturbations in the 3D case leads to  two new difficulties. The first, with respect to the 2D case, comes from the presence of continuous spectrum, then the spectral study involves resonances and some non-selfadjoint aspects. The second difficulty, with respect to the potential perturbations, is due to the fact that the perturbed and the unperturbed operators are not defined on the same space. In order to overcome this difficulty, we introduce an appropriate perturbation $V^{l}$, $l = \infty, \gamma$, of $H_0^{-1}$  on $L^2(\R^3)$ in such a way that the concentration of resonances 
 of $
 H_{\Omega}^{l}$, $l = \infty, \gamma$ at $\Lambda_q$ is reduced to the accumulation of "{\it Birman-Schwinger singularities}" of $H_0^{-1}-V^{l}$ at $\frac1{\Lambda_q}$ in the sense that $\frac1{z}$ is a "{\it Birman-Schwinger singularity}" of $H_0^{-1}-V^{l}$ if 
 $1$ is an eigenvalue of 
$$
B^l(z):= \hbox{sign} (V^l) \, |V^l |^\frac12 \,   \Big( H_0^{-1} -  \frac1{z}  \Big)^{-1}  \, |V^l |^\frac12$$
$$=
\hbox{sign} (V^l) \, |V^l |^\frac12 \, z H_0 (H_0-z)^{-1} |V^l |^\frac12 = z V^l + z^2 \hbox{sign} (V^l) \, |V^l |^\frac12 \, (H_0-z)^{-1} |V^l |^\frac12
$$
(see  Section \ref{s3} and in particular Proposition \ref{p1}). Then the main tool of our proof is an abstract result of \cite{BoBrRa14_01} (see Section \ref{S:proof} and especially Proposition \ref{P:P1}).
\Bk 
%
%
%
%
%
%
%
%
%
%
%
%
%
%

\section{Magnetic resonances for the exterior problems}\label{s3}

\quad In this section we reduce the study of the operators $H_\om^\infty$ and $H_\om^\gamma$ near the Landau levels to some compact perturbations, of fixed sign, of $H_0^{-1}$. We follow ideas developped in \cite{PuRo07} and \cite{GoKaPe13} for the eigenvalues of the $2D$ Schr\"odinger operators and  give a charaterisation of the resonances which will allow to apply (in Section \ref{S:proof}) a general result of \cite{BoBrRa14_01}.

\subsection{Auxiliary operators}\label{ss:aux}

\quad By identification of $L^2(\R^3)$ with  $L^2(\om)\oplus L^2(K)$, we consider  the following operator in  $L^2(\R^3)$:
\begin{equation}\label{defhtilde}
{\Tilde H}^\gamma:= H_\om^\gamma  \oplus H_K^{-\gamma}   \qquad \hbox{on} \qquad    Dom(H_\Omega^\gamma) \oplus Dom(H_K^{-\gamma}),
\end{equation}
where $H_K^{-\gamma}$ is the Robin operator in $K$. Namely, $H_K^{-\gamma}$ is the self-adjoint operator associated to the closure of the quadratic form $Q_{K}^{-\gamma}$ defined by \eqref{q1}, by replacing $\gamma$ and $\Omega$ with $-\gamma$ and $K$ respectively.

Without loss of generality, we can assume that $H_\om^\gamma$, $H_K^{-\gamma}$ and $H_\om^\infty$ are positive and invertible (if not it is sufficient to shift their by the same constant), and  we introduce 
\bel{defVr}
V^\gamma:= H_0^{-1} - ({\Tilde H}^\gamma)^{-1}= H_0^{-1} - (H_\om^\gamma )^{-1} \oplus (H_K^{- \gamma} )^{-1}.
\ee
\bel{defVd}
V^\infty:= H_0^{-1} - (H_\om^\infty )^{-1} \oplus 0.
\ee

On one hand, thanks to the choice of the boundary condition in $K$ (with $-\gamma$), the quadratic form associated to ${\Tilde H}^\gamma$ is given by 
\begin{equation}\label{q2}
{\Tilde Q}^{\gamma}(u_{\Omega},u_{K}) =Q_{\Omega}^{\gamma}(u_{\Omega}) + Q_{K}^{-\gamma}(u_{K}) = \int_{\R^{3}} \big\vert \nabla^A u \big\vert^{2} dx.
\end{equation}
Thus, $d({\Tilde Q}^{\gamma})$, the domain of the quadratic form associated to ${\Tilde H}^\gamma$ contains $H^1_A(\R^3)$ the domain of $Q_{0}$, the quadratic form associated to $H_0$:
$$ Q_{0}(u)=\int_{\R^{3}} \big\vert \nabla^A u \big\vert^{2} dx,$$
and ${\Tilde Q}^{\gamma}$ coincide with $Q_{0}$ on $H^1_A(\R^3)$.

On the other hand by extending by $0$ the functions of the domain of the quadratic form $Q_{\Omega}^{\infty}$ (see \eqref{q3}) we can embed $d( Q_{\Omega}^{\infty})$ in $H^1_A(\R^3)$ with $Q_{\Omega}^{\infty}$ coinciding with $Q_{0}$ on $d( Q_{\Omega}^{\infty})$.

From the previous properties, according to Proposition 2.1 of \cite{PuRo07}, we deduce that  $V^\gamma$ (defined by \eqref{defVr}) is a non positive operator and   $V^\infty$ (defined by \eqref{defVd}) is non negative in $\Lc(L^2(\R^3))$.

\subsection{Decreasing  and compact properties of the perturbations of fixed sign}

\begin{lemma}\label{lemVM}
The operators $V^\infty$ and $V^\gamma$ defined by (\ref{defVd}) and (\ref{defVr}) are respectively non negative and non positive compact operators in $\Lc(L^2(\R^3))$. Moreover, there exists  $M^\infty$ and $M^\gamma$ compact operators  in $\Lc(L^2(\R^3))$ such that 
\bel{decompositionV}
V^\infty= M^\infty M_\infty, \qquad V^\gamma= - M^\gamma M_\gamma; \qquad M_l:= (M^l)^*, l=\gamma, \infty,
\ee
with $M_l$, $l = \gamma, \infty$ bounded from $e^{\epsilon \langle x_3 \rangle} L^2(\R^3)$ into $L^2(\R^3)$, $0< \epsilon < \sqrt{b}$.

\end{lemma}

\begin{proof}
Since $V^\infty$ and $(- V^\gamma)$ are  non negative bounded operators on $L^2(\R^3)$ (see Subsection \ref{ss:aux} above) there exists bounded operators $M_l$, $l=\gamma, \infty$ such that 
$V^\infty = (M_\infty)^* \, M_\infty$, and $V^\gamma = - (M_\gamma)^* \, M_\gamma$. For example, we can take $M_l=(M_l)^*=|V_l|^{\frac12}$, but sometimes other choices could be more convenient (see remark \ref{rem:Mlchoice}).

For all $f, g \in L^2(\R^3)$, let us introduce $v=H_0^{-1} f$, $u_{\Omega,l}=(H_\Omega^l)^{-1}(g_{|\Omega})$ and $u_{K}=(H_K^{-\gamma})^{-1}(g_{| K})$. 
By definition of $V^l$, $ l=\gamma, \infty$, we have:
$$\langle f , V^\infty g \rangle_{L^2(\R^3)}= \langle H_0 v, \, \big(H_0^{-1} - (H_\om^\infty )^{-1} \oplus 0 \big) \big( H_\om^\infty u_{\Omega,\infty} \oplus  g_{| K}  \big) \rangle$$
$$= \langle v, \,  H_\om^\infty u_{\Omega,\infty} \oplus g_{| K} \rangle  - \langle H_0 v, \,    u_{\Omega,\infty } \oplus 0 \rangle$$
$$= -  \int_\om v \, \overline{(\nabla^A)^2 u_{\om,\infty}} dx +  \int_\om (\nabla^A)^2 v \, \overline{u_{\om,\infty}} dx + \int_K v \, \overline{g_{| K}} dx,$$
and
$$\langle f , V^\gamma g \rangle_{L^2(\R^3)} = \langle H_0 v, \, \big(H_0^{-1} - ({\Tilde H}^\gamma)^{-1}\big) \big(H_\om^\gamma u_{\Omega,\gamma} \oplus H_K^{-\gamma} u_{K}  \big) \rangle$$
$$= \langle v, \,  \big(H_\om^\gamma u_{\Omega,\gamma} \oplus H_K^{-\gamma} u_{K}  \big) \rangle  - \langle H_0 v, \,    u_{\Omega,\gamma} \oplus  u_{K} \rangle$$
$$=  -  \int_\om v \, \overline{(\nabla^A)^2 u_{\om,\gamma}} dx  +  \int_\om (\nabla^A)^2 v \, \overline{u_{\om,\gamma}} dx 
-  \int_K v \, \overline{(\nabla^A)^2 u_{K}} dx  +  \int_K (\nabla^A)^2 v \, \overline{u_{K}} dx. $$

Then by integration by parts, from the boundary conditions $u_{\om,\infty \vert \Sigma}= 0 $ and $\dbdR u_{K}=0 = \dbdR u_{\Omega,\gamma}$, we deduce the equalities
\bel{ExpVinfty}
\langle f , V^\infty g \rangle_{L^2(\R^3)}  = \int_K v_{| K} \, \overline{g_{| K}} dx + \int_\Sigma \Gamma_0(v) \,  \overline{\dbdN u_{\om,\infty}}d \sigma,
\ee
\bel{ExpVgamma}
\langle f , V^\gamma g \rangle_{L^2(\R^3)}  = - \int_\Sigma \dbdR v \,\overline{\Gamma_0 \big(  u_{\om,\gamma}\big)- \Gamma_0 \big(u_{K}\big)}d \sigma,
\ee
where  $\Gamma_0: H^{s}(\bullet) \longrightarrow H^{s-\frac12}(\Sigma)$, for $\bullet = \om$, $K$,  and $s\geq \frac12$, is the trace operator on $\Sigma$. 
In the notation of this operator, we omit the dependence on $K$ or $\om$ because either it is indicate on the functions on which it is applied, or the functions are smooth near $\Sigma$.
In particular, due to the regularity properties of $v=H_0^{-1} f$ near $\Sigma$, the functions $\Gamma_0(v)$ and $ \dbdR v $ are well defined. 

In other words, we have
\bel{ExpVinftyb}
\langle f , V^\infty g \rangle_{L^2(\R^3)} 
=\langle  (H_0^{-1} f)_{|K} , g_{| K} \rangle_{L^2(K)} + \langle \Gamma_0(H_0^{-1} f )\, , \, \dbdN (H_\Omega^\infty)^{-1}(g_{|\Omega}) \rangle_{L^2(\Sigma)},
\ee

\bel{ExpVgammab}
\langle f , V^\gamma g \rangle_{L^2(\R^3)}  = - \langle \dbdR (H_0^{-1} f ), \Gamma_0  (H_\Omega^\gamma)^{-1}(g_{|\Omega}) -  \Gamma_0 (H_K^{-\gamma})^{-1}(g_{|K}) \rangle_{L^2(\Sigma)}.
\ee 
Exploiting that the domains of the operators contain $H^1_{loc}$ and the compacity of the domains $K$ and $\Sigma$, we deduce that  $V^\infty$ and $V^\gamma$ are compact operators in $L^2(\R^3)$.

At last, in order to prove that $M_\infty$ and  $M_\gamma$ are bounded from $e^{\epsilon \langle x_3 \rangle} L^2(\R^3)$ into $L^2(\R^3)$, let us prove that for $l=\gamma, \infty$, $e^{\epsilon \langle x_3 \rangle} V^l e^{\epsilon \langle x_3 \rangle} $ is bounded in $L^2(\R^3)$.

Clearly, in the relations \eqref{ExpVinfty} and \eqref{ExpVgamma}, $v$ can be replaced by $\chi_3 v$ for any $\chi_3 \in C^\infty_c(\R_{x_3})$ equals to $1$ on 
\bel{defIK}
I_K:=\overline{\cup_{(x_1,x_2) \in \R^2} \{x_3 ; \; (x_1,x_2,x_3) \in K \}},
\ee
 ( $\chi_3 = 1$ near $K$ is sufficient)
  and consequently, for $ l=\gamma, \infty$, we have 
$\langle f , V^l g \rangle= \langle f_{\chi_3} , V^l g \rangle$ with $ f_{\chi_3}= H_0(\chi_3 v)= (H_0 \chi_3 H_0^{-1} )f$. Thus, 
$V^l = H^{-1}_0 \chi_3 H_0 V^l$ and taking the adjoint relation we deduce:
$$
V^l= V^lH_0 \chi_3 H^{-1}_0  = H^{-1}_0 \chi_3 H_0 V^lH_0 \chi_3 H_0^{-1}
$$ 
By using the orthogonal decomposition of $H_0^{-1}$:
$$H_0^{-1}= \sum_{q \in \N} p_q \otimes (D^2_{x_3} + \l_q)^{-1}$$
with $(D^2_{x_3} + \l_q)^{-1}$ having the integral kernel $\frac{e^{-\sqrt{\l_q} \vert x_3 -x_3' \vert}}{2 \sqrt{\l_q} } $, we deduce that  $H_0 \chi_3 H_0^{-1}= \chi_3 + [D_{x_3}^2, \chi_3] H_0^{-1}$ is bounded from $e^{\epsilon \langle x_3 \rangle} L^2(\R^3)$ into $L^2(\R^3)$ for $0< \epsilon < \sqrt{b}$ and then so is $M_l$, $l = \gamma, \infty$.  
 
 \end{proof}

\begin{remark}\label{rem:Mlchoice}
As written in the above proof, we can take $M_l=(M_l)^*=|V_l|^{\frac12}$, but sometimes other choices could be more convenient. In particular in order to reduce our analyse to the boundary $\Sigma$, it could be interesting to consider operator $M_l$ from $L^2(\R^3)$ into $L^2(\Sigma)$ by exploiting the link with the Dirichlet-Neumann and Robin-Dirichlet operators (see Subsection \ref{ss3,3}).
\end{remark}


\subsection{Relation with Dirichlet-Neumann and Robin-Dirichlet operators}\label{ss3,3}

\quad Taking $g=f=H_0 v$ and by introducing $w_{\om,\infty}:= v_{|\Omega} - u_{\om,\infty}$ in \eqref{ExpVinfty}, we obtain:

\bel{Vinftyff}
\langle H_0 v , V^\infty H_0 v \rangle_{L^2(\R^3)}  = \int_K v_{| K} \, \overline{f_{| K}} dx + \int_\Sigma \Gamma_0(v) \,  \overline{\dbdN v_{|\Omega}}d \sigma -
\int_\Sigma \Gamma_0(v) \,  \overline{\dbdN w_{\om,\infty}}d \sigma,
\ee
with $w_{\om,\infty}$ satisfying
\bel{eq100}
 \left\{
 \begin{array}{rclccc}
 (\nabla^A)^2 w_{\om,\infty} & = &0 & \hbox{ in } & \om\\
\Gamma_0( w_{\om, \infty})&= &\Gamma_0(v).
 \end{array}
 \right.
 \ee
 
On the same way, from \eqref{ExpVgamma}, we have   
 \bel{ExpVgammaff}
\langle H_0 v , V^\gamma H_0 v \rangle_{L^2(\R^3)}  = - \int_\Sigma \dbdR v \, \overline{\Gamma_0\big(  w_{K,\gamma}\big)- \Gamma_0\big(w_{\om,\gamma}\big)}d \sigma,
\ee
with  $w_{\om,\gamma}=v_{|\om}-u_{\om,\gamma}$ and $w_{K,\gamma}=v_{|K}-u_{K}$ satisfying for $\bullet=\om, K$
\bel{eq101}
 \left\{
 \begin{array}{rclccc}
 (\nabla^A)^2 w_{\bullet,\gamma}& =& 0 & \hbox{ in } & \bullet\\
 \dbdR w_{\bullet,\gamma}& = & \dbdR v.
 \end{array}
 \right.
 \ee

 Consequently, $V^\infty$ and $V^\gamma$ are related to the {\it Dirichlet-Neumann} and {\it Robin-Dirichlet} operators:
 \bel{eq801}
\langle H_0 v , V^\infty H_0 v \rangle_{L^2(\R^3)}  = 
\langle  v_{|K} , (H_0 v)_{| K} \rangle_{L^2(K)} + \langle \Gamma_0( v_{|\om})\, , \, \dbdN (v_{|\om})-{\bf DN_\Omega} (\Gamma_0 v_{|\om}) \rangle_{L^2(\Sigma)},
\ee

\bel{eq811}
\langle H_0 v, V^\gamma H_0 v \rangle_{L^2(\R^3)}= - \langle \dbdR v \, , \,( {\bf RD_K}-{\bf RD_\Omega}) \dbdR v \rangle_{L^2(\Sigma)}.
\ee
{ For the definition and elliptic properties of these operators (on some subspaces of finite codimension), we refer to Proposition \ref{defDRRD}.

{

\subsection{Definition and characterisation of the resonances}

\quad Let us introduce, for $\im (z) >0$, the bounded operators 
\bel{defRNR}
{\Tilde R}^\gamma(z)=({\Tilde H}^\gamma-z)^{-1}=({H}^\gamma_\Omega-z)^{-1}\oplus ({H}^{-\gamma}_K-z)^{-1}  \quad    \textup{and} \quad  {\Tilde R}^\infty(z)=({H}^\infty_\Omega-z)^{-1}\oplus 0.
\ee

\begin{proposition}\label{p1}
For $l=\gamma, \infty$, the operator-valued function 
$${\Tilde R}^l(z): e^{- \epsilon \langle x_3 \rangle} L^2(\R^3) \longrightarrow e^{\epsilon \langle x_3 \rangle}L^2(\R^3)$$
has a meromorphic extension \big(also denoted ${\Tilde R}^l(\cdot)$\big) from the open upper half plane to ${\mathcal M}_\epsilon$, $\epsilon <\sqrt{b}$. Moreover, the following assertions are equivalent:
\begin{enumerate} 
\item{z is a pole of ${\Tilde R}^l$ in $\Lc\Big(e^{- \epsilon \langle x_3 \rangle} L^2(\R^3), e^{\epsilon \langle x_3 \rangle}L^2(\R^3) \Big)$},
\item{z is a pole of $M_l{\Tilde R}^lM^l$ in $\Lc\Big( L^2(\R^3)\Big)$},
\item{ $-1$ is an eigenvalue of $\varepsilon(l) B^l(z)$ with
\bel{defT}
B^l(z):= z M_lH_0 R_0(z) M^l= z M_lM^l + z^2 M_l R_0(z) M^l,
\ee
where $\varepsilon(\infty)=1$, $\varepsilon(\gamma)=-1$. 
}
\end{enumerate}
\end{proposition}
\begin{proof}
For $\im (z) >0$ and $l=\gamma, \infty$ we have the resolvent equation:
$$ {\Tilde R}^l(z) \Big( I + z V^l H_0 (H_0-z)^{-1}\Big) ={\Tilde R}^l(0) H_0 (H_0-z)^{-1}=(H_0-z)^{-1} -  V^l H_0 (H_0-z)^{-1}.$$
Let us denote by $e_{\pm}$ the multiplication operator by  $ \textup{e}^{\pm \epsilon \langle x_{3} \rangle}$. 
Then, by introducing $e_\pm$, and writing $H_0 (H_0-z)^{-1}= I + z (H_0-z)^{-1}$, we have
$$e_- {\Tilde R}^l(z) e_-  \,  \Big( I + z e_+ V^l e_- + z^2 e_+ V^l  (H_0-z)^{-1}e_- \Big) =$$
$$ e_- (H_0-z)^{-1}e_- - e_- V^l e_-  - z e_- V^l  (H_0-z)^{-1}e_-.$$
Since $e_- ({H}_0 -z)^{-1} e_-$ admits a holomorphic extension from the open upper half plane to ${\mathcal M}_\epsilon$ \big(see Proposition 1 of \cite{BoBrRa07_01}\big), according to Lemma \ref{lemVM}, $e_+ V^l  (H_0-z)^{-1}e_-$  and $e_- V^l  (H_0-z)^{-1}e_-$ 
can be holomorphically extended to ${\mathcal M}_\epsilon$. Then from the Fredholm analytic Theorem we deduce the  meromorphic extension of $z \mapsto {\Tilde R}^l(z)$ from the open upper half plane to ${\mathcal M}_\epsilon$. 

Moreover, by writing 
$$M_l{\Tilde R}^l(z) M^l = (M_l e_+) \, e_- {\Tilde R}^l(z)  e_- \, ( e_+ M^l), $$
we show the holomorphic extension of $M_l{\Tilde R}^lM^l$ in $\Lc\Big( L^2(\R^3)\Big)$ with poles among those of $e_- {\Tilde R}^l  e_- $. Conversely, according to the following resolvent equation, the poles of $e_- {\Tilde R}^l  e_- $ are those of $M_l{\Tilde R}^lM^l$ in $\Lc\Big( L^2(\R^3)\Big)$:
$${\Tilde R}^l(z) = (H_0-z)^{-1} - V^l  H_0 (H_0-z)^{-1} - z H_0(H_0-z)^{-1}  \,  {\Tilde R}^l(0) \, V^l \, H_0 (H_0-z)^{-1} $$
$$+ z^2  H_0(H_0-z)^{-1} \, V^l \, {\Tilde R}^l(z) \, V^l \, H_0 (H_0-z)^{-1}.$$
We conclude the proposition from the equation:
\bel{eqreso}
\Big( I  + \varepsilon(l) M_l \Big(\frac{1}{z} - H_0^{-1}\Big)^{-1} M^l\Big) \,\Big( I  -  \varepsilon(l) M_l  \Big(\frac{1}{z} -{\Tilde R}^l(0)\Big)^{-1} M^l\Big) = I, 
\ee
and using that
$$M_l \Big(\frac{1}{z} - H_0^{-1}\Big)^{-1}M^l= z M_l H_0(H_0 -z)^{-1} M^l = z M_l M^l+ z^2 M_l (H_0 -z)^{-1} M^l,$$
$$ M_l \Big (\frac{1}{z} -{\Tilde R}^l(0)\Big)^{-1} M^l=z M_l M^l+ z^2 M_l {\Tilde R}^l(z) M^l.$$

\end{proof}
By definition of ${\Tilde R}^l$, $l=\infty, \gamma$, (see \eqref{defRNR}), we also have:
\begin{corollary}\label{Cor:PoleReso}
For $l=\gamma, \infty$, the operator-valued function 
$$({H}^l_\Omega-z)^{-1}: e^{- \epsilon \langle x_3 \rangle} L^2(\Omega) \longrightarrow e^{\epsilon \langle x_3 \rangle}L^2(\Omega)$$
has a meromorphic extension, denoted $R^l_\om$, from the open upper half plane to ${\mathcal M}_\epsilon$, $\epsilon <\sqrt{b}$. 

Moreover, according to their multiplicities (i.e. the rank of their residues), the poles of $R^\infty_\om$
coincide with the poles of ${\Tilde R}^\infty$ and the poles of $R^\gamma_\om$
are those of ${\Tilde R}^\gamma$ excepted the eigenvalues of ${H}^{-\gamma}_K$.
\end{corollary}

}

\section{Outline of proofs}\label{S:proof}

\quad In order to prove our main results, in this section, let us begin by recalling some auxiliary results concerning characteristic values of holomorphic operators due to Bony, the first author and Raikov \cite{BoBrRa14_01}. Then we will apply these auxiliary results to our problem and prove the main results.

\subsection{Auxiliary results}\label{subP1}

\quad Let $\mathcal{D}$ be a domain of $\mathbb{C}$ containing zero and let us consider an holomorphic operator-valued function $A : \mathcal{D} \longrightarrow S_{\infty}$, where $S_{\infty}$ is the class of compact operators in a separable Hilbert space.

\begin{definition}\label{defp1}
For a domain $\Delta \subset \mathcal{D} \setminus \lbrace 0 \rbrace$, a complex number $z \in \Delta$ is a \textit{characteristic value} of $z \mapsto I - \frac{A(z)}{z}$ if the operator $I - \frac{A(z)}{z}$ is not invertible. The multiplicity of a 
characteristic value $z_{0}$ is defined by
\begin{equation}\label{eqp2}
\textup{mult}(z_{0}) := \frac{1}{2i \pi}\textup{tr} \Big( \int_{\mathcal{C}} \big(- \frac{A(z)}{z}\big)^\prime \, \big(I - \frac{A(z)}{z}\big)^{-1} dz\Big),
\end{equation}
where $\mathcal{C}$ is a small contour positively oriented containing $z_{0}$ as the unique point $z$ satisfying $\big(I - \frac{A(z)}{z}\big)$ is not invertible.
\end{definition}

Let us denote by $\mathcal{Z}(\Delta,A)$, the set of the  \textit{characteristic values} of $\big(I - \frac{A(z)}{z}\big)$ inside $\Delta$:
$$
\small{ \mathcal{Z}(\Delta,A):= \left\lbrace z \in \Delta : I - \frac{A(z)}{z} \hspace{0.8mm} \textup{is not invertible} \right\rbrace}.
$$ 
If there exists $z_{0} \in \Delta$ such that $I - \frac{A(z_{0})}{z_{0}}$ is not invertible, then $\mathcal{Z}(\Delta,A)$ is a discrete set (see e.g. \cite[Proposition 4.1.4]{GoSi71_01}). 

Assume that $A(0)$ is self-adjoint and for $T$ a compact self-adjoint operator, let us introduce the counting function
\begin{equation}\label{eqP0}
n(r,T) := \textup{Tr} \hspace{0.6mm} \one_{[r,+\infty)}(T),
\end{equation}
 the number of eigenvalues of the operator $T$ lying in the interval $[r,+\infty) \subset \mathbb{R}^{\ast}$, and counted with their multiplicity. Denote by $\Pi_{0}$ the orthogonal projection onto $\textup{ker}  \hspace{0.6mm} A(0)$.
 
As consequence of \cite[Corollary 3.4, Theorem 3.7, Corollary 3.11]{BoBrRa14_01}, 
 we have the following result which states that  the \textit{characteristic values} of $z \mapsto I - \frac{A(z)}{z}$ are localized near the real axis where $A(0)$ has its spectrum, and the distribution of the \textit{characteristic values} near $0$ is governed by the distribution of the eigenvalues of $A(0)$ near $0$. 

\begin{proposition}\label{P:P1} 
\textup{\cite
{BoBrRa14_01}}

 Let $A$ be as above and $I - A^\prime(0) \Pi_{0}$ be invertible. Assume that $\Delta \Subset \mathbb{C} \setminus \lbrace 0 \rbrace$ is a bounded domain with smooth boundary $\partial \Delta$ which is transverse to the real axis at each point of $\partial \Delta \cap \mathbb{R}$. We have:

$\textup{\textbf{(i)}}$ The characteristic values $z \in \mathcal{Z}(\Delta,A)$ near $0$ satisfy $\vert \Im(z) \vert = o(\vert z \vert)$.

$\textup{\textbf{(ii)}}$ If the operator $A(0)$ has a definite sign $(\pm A(0) \geq 0)$, then the characteristic values $z$ near $0$ satisfy $\pm \Re(z) \geq 0$.

$\textup{\textbf{(iii)}}$ For $\pm A(0) \geq 0$, if the counting function of $A(0)$ satisfies:
$$n(r, \pm A(0))= c_0  \frac{\vert \ln r \vert}{\ln \vert \ln r \vert} (1+o(1)), \qquad r \searrow 0,$$
then, for $r_0 > 0$ fixed, the counting function of the characteristic values near $0$ satisfies:
$$\# \big\lbrace z \in \mathcal{Z}(\Delta,A); \; r < \vert z \vert < r_0 \big\rbrace = 
c_0  \frac{\vert \ln r \vert}{\ln \vert \ln r \vert} (1+o(1)), \qquad r \searrow 0.$$
\end{proposition}

{\Rd

}

\subsection{Preliminary results}\label{subP2}

\quad In this subsection we apply the previous abstract results to our problem.





\begin{proposition}\sl \label{c8}
Fix  $q \in \N$. Then $z_q(k)=\Lambda_q + k^2$, $0<|k| \ll 1$ is a
pole of ${\Tilde R}^l$ \big(defined by \eqref{defRNR}\big) in $\mathcal{L}( e^{- \epsilon \langle x_3 \rangle} L^2 , e^{\epsilon \langle x_3 \rangle}L^2)$ if and only if 
\bel{defel}
I - \varepsilon(l) \frac{A_q^l(ik)}{ik} \textup{ is not invertible }, 
\qquad
\varepsilon(l) := 
\begin{cases}
\hspace{0.3cm} 1 & \text{if} \quad l = \infty \\
-1 & \text{if} \quad l = \gamma
\end{cases},                  \qquad          
\ee
where $z \mapsto A_q^l(z) \in S_\infty(L^2(\R^3))$ is the holomorphic operator-valued function given by
\bel{defAql}
A_q^l(ik)= z_q(k)^2 M_l(p_q \otimes r(ik) )M^l - ik M_l\Big(z_q(k)  + z_q(k)^2  R_0(z_q(k))(I- p_q \otimes I_3)  \Big)M^l,
\ee
with $r (z)$ the integral operator in $L^2(\R_{x_3})$ whose  integral kernel is $\frac12 e^{z \vert
x_3-x^\prime_3\vert}$, $x_3, x_3'\in \R$.

In particular, $ A_q^l(0)= \Lambda_q^2 M_l(p_q \otimes r(0) )M^l$ is a non negative compact operator whose counting function satisfies
\bel{A0T}
n \big( r, A_q^l(0) \big) = n \big( r, T_{q}^{l} \big); \qquad T_{q}^{l} := \varepsilon(l) \Lambda_q^2 \, p_{q}W^{l}p_{q}
\ee
with $W^{l}$ defined on $L^{2}(\R^{2})$ by 
\begin{equation}\label{id4}
(W^l f^\perp) (x_{\perp})= \frac{1}{2} \int_{\R_{x_3}} \Big(V^l (f^\perp \otimes 1_{\mathbb{R}} ) \Big)(x_\perp,x_3) dx_3, 
\end{equation}

\end{proposition}

%
%
\begin{proof}
From Proposition \ref{p1}, $z_q(k) = \Lambda_q + k^2$  is a pole of ${\Tilde R}^l(.)$ if and only if  $I + \varepsilon(l)  B^l(\Lambda_q + k^2)$ is not invertible with 
$$B^l(\Lambda_q + k^2):= z_q(k) M_lM^l + z_q(k)^2 M_l R_0(\Lambda_q + k^2) M^l.$$
We split the sandwiched resolvent $M_l R_0(z) M^l$ into two parts as follows
\begin{equation}\label{dec1}
M_l R_0(z) M^l = M_l R_0(z) (p_q \otimes I_3 ) M^l + M_l R_0(z) (I - p_q \otimes I_3) M^l.
\end{equation}
For $z_q(k)=\Lambda_q+k^2$ in the resolvent set of the operator $H_0$, we have
$$
(H_0 - \Lambda_q-k^2)^{-1} = \sum_{j \in \N} p_{j} \otimes \big( D_{x_3}^{2} + \Lambda_{j} - \Lambda_q - k^2 \big)^{-1}.
$$
Hence by definition of $p_q$, $M_l R_0(\Lambda_q+k^2) (I - p_q \otimes I_3) M^l$ is holomorphic near $k=0$ (for more details, see the proof of Proposition \ref{propMeroR0}). Furthermore, for $k$ chosen such that $\Im (k) > 0$, we have
\begin{equation}\label{dec2}
M_l R_0(z_q(k)) (p_q \otimes I_3) M^l = M_l p_q \otimes \big( D_{x_3}^2 - k^{2} \big)^{-1} M^l 
= -\frac{ M_l (p_q \otimes r(ik)) M^l}{ik},
\end{equation}
where $r(z)$ is the integral operator introduced above. Hence \eqref{defel} and \eqref{defAql}
 hold because 
 $$B^l(\Lambda_q + k^2) = - \frac{A_q^l(ik)}{ik}. $$
Let us compute the operator $A_q^l(0)$ for $l=\gamma, \infty$. We have
\begin{equation}\label{opA_0}
A_q^l(0) = \Lambda_q^2 M_l (p_q \otimes r(0)) M^l, \qquad l = \gamma, \infty,
\end{equation}
where $r(0)$ is the operator acting from $e^{- \epsilon \langle x_3 \rangle} L^{2}(\R)$ into $e^{ \epsilon \langle x_3 \rangle} L^{2}(\R)$ with integral kernel given by the constant function $\frac{1}{2}$. 

Now from Lemma \ref{lemVM}, it follows that there exists a bounded operator $\mathcal{M}_{l}$ on $L^2(\R^3)$ such that $M_{l} = \mathcal{M}_{l} e^{- \epsilon \langle x_3 \rangle}$ for $l=\gamma, \infty$. Recalling that $e_\pm$ is the multiplication operator by $e_{\pm} :=  \textup{e}^{\pm \epsilon \langle x_{3} \rangle}$, it can be easily checked that 
\begin{equation}\label{id1}
M_l (p_q \otimes r(0)) M^l = \frac{M_l e_{+} (p_q \otimes c^{\ast} c) e_{+} M^l}{2}
= B_{q,l}^{\ast} B_{q,l},
\end{equation} 
where $c : L^{2}(\mathbb{R}) \longrightarrow \mathbb{C}$ is the operator defined by $c(u) := \langle u,e_{-} \rangle$, so that $c^{\ast} : \mathbb{C} \longrightarrow L^{2}(\mathbb{R})$ is given by $c^{\ast}(\lambda) = \lambda e_{-}$, and 
\begin{equation}\label{id2}
B_{q,l} := \frac{1}{\sqrt{2}} (p_q \otimes c) e_{+} M^l, \qquad l=\gamma, \infty.
\end{equation} 
More explicitly, the operator $B_{q,l}$ satisfies $B_{q,l} : L^{2}(\R^{3}) \longrightarrow L^{2}(\R^{2})$ with 
$$
(B_{q,l}\varphi)(x_{\perp}) = \frac{1}{\sqrt{2}} \int_{\R^{3}} {\mathcal P}_{q,b}(x_{\perp},x_{\perp}^\prime) 
(M^{l}\varphi)(x_{\perp}^\prime,x_{3}^\prime) dx_{\perp}^\prime dx_{3}^\prime,
$$
where ${\mathcal P}_{q,b}(\cdot,\cdot)$ is the integral kernel of $p_q$ given by:
\begin{equation}\label{ker1}
{\mathcal P}_{q,b}(x_\perp,x_\perp^\prime)=\frac{b}{2\pi} L_q \left( \frac{b \vert x_\perp - x_\perp^\prime \vert^2}{2}\right)
\exp \Big( -\frac{b}{4} \big( \vert x_\perp - x_\perp^\prime \vert^2 + 2i(x_1x_2^\prime-x_1^\prime x_2) \big) \Big),
\end{equation}
with $x_\perp = (x_{1},x_{2}), x_\perp^\prime = (x_1^\prime,x_2^\prime) \in \R^2$; here $L_q (t) : = \frac{1}{q !}
e^t \frac{d^q ( t^q e^{-t} )}{d t^q}$ are the Laguerre polynomials.

The adjoint operator $B_{q,l}^{\ast} : L^{2}(\R^{2}) \longrightarrow L^{2}(\R^{3})$ satisfies
$$
(B_{q,l}^{\ast}f^{\perp})(x_{\perp},x_{3}) = \frac{1}{\sqrt{2}} M_{l} (p_qf^{\perp}  \otimes 1_{\mathbb{R}}) (x_{\perp},x_{3}),
$$
where $(p_qf^{\perp} \otimes 1_{\mathbb{R}})(x_{\perp},x_{3}) = p_qf^{\perp}(x_{\perp})$ is constant with respect to $x_{3}$. 
Thus, from \eqref{opA_0} and \eqref{id1}, $A_q^l(0)=\Lambda_q^2 B_{q,l}^{\ast} B_{q,l}$ is a positive compact self-adjoint operator with the non zero eigenvalues equal to those of  $\Lambda_q^2 B_{q,l} B_{q,l}^{\ast} : L^{2}(\R^{2}) \longrightarrow L^{2}(\R^{2})$ given by
\begin{equation}\label{id3}
\Lambda_q^2 B_{q,l} B_{q,l}^{\ast} =\varepsilon(l)  \Lambda_q^2  p_{q}W^{l}p_{q}=T^l_q, \qquad l=\gamma, \infty,
\end{equation}
where  $W^{l}$ is defined on $L^{2}(\R^{2})$ by \eqref{id4}. Here, we have used that $M^lM_l = \varepsilon(l) V^l$ (see Lemma \ref{lemVM}).

 Then
\begin{equation}\label{id5}
\small{n \big( r, A_q^l(0) \big) = n \big( r, \Lambda_{q}^{2} B_{q,l}^{\ast} B_{q,l} \big)
= n \big( r, \Lambda_{q}^{2} B_{q,l} B_{q,l}^{\ast} \big) = n \big( r, T_{q}^{l} \big).}
\end{equation}
\end{proof}

\begin{remark}
Here, we omit the proof that the multiplicity of the pole $z_q(k)$ of ${\Tilde R}^l$ (i.e. the rank of the residue) coincides with the multiplicity of $k$ as a characteristic value of $k \mapsto I -  \epsilon(l) \frac{A_q^l(ik)}{ik}$ \big(see \eqref{eqp2}\big).
It is closely related to the proof of Proposition 3 of \cite{BoBrRa07_01}.
\end{remark}

In order to  analyse the counting function $n \big( r, A_q^l(0) \big) = n \big( r, T_{q}^{l} \big)$ for $l = \gamma, \infty$, we will need the following result.

\begin{proposition}\label{propT1}
For all $q \in \N$, there exists $L_q$ a finite codimension subspaces of ker$(H_{\text{Landau}}-\Lambda_q)$ and there exists $C>1$, such that for compacts domains $K^\perp_0 \subset  K^\perp_1 \subset \R^2$ and compact intervals $I_0\subset I_1$ satisfying $K^\perp_0 \times I_0 \subset  K \subset K^\perp_1\times I_1$ with $\partial (K^\perp_i \times I_i) \cap \partial K= \emptyset$, $i=0,1$, we have for any $f^\perp \in L_q$, 
\bel{equivT}
\frac1{C} \langle f^\perp , p_q \one_{K^\perp_0} p_q f^\perp \rangle_{L^2(\R^2)} \leq  \langle f^\perp , T_q^l f^\perp \rangle_{L^2(\R^2)} 
\leq C  \langle f^\perp , p_q \one_{K^\perp_1} p_q f^\perp \rangle_{L^2(\R^2)},
\ee
where $T_q^l$ is defined by \eqref{A0T},  $ l=\infty, \gamma$.
\end{proposition}

The proof of Proposition \ref{propT1} will be given in Section \ref{S:sec6} by introducing elliptic pseudo-differential operators on the boundary $\Sigma = \partial \Omega = \partial K$.

\subsection{Proof of Theorem \ref{theo1}}

\quad First, near $\Lambda_q$, Corollary \ref{Cor:PoleReso} allows to reduce the study of the resonances of $H_{\Omega}^{l}$, $z_{q}(k) = \Lambda_{q} + k^2$, $|k|\leq r_0$ ($r_0$ sufficiently small) to the poles of ${\Tilde R}^l$, because they coincide for $l= \infty$ and for $l=\gamma$ there is only a finite number of eigenvalues of $H^{-\gamma}_K$ near $\Lambda_q$. Then, thanks to Proposition \ref{c8}, modulo a finite set ${\mathcal F}_q$, the  resonances near $\Lambda_q$ are related to the characteristic values of $z  \mapsto I - \frac{A_{q}^{l}(\varepsilon(l)z)}{z}$:
\bel{etape1}
 \{ z_{q}(k) = \Lambda_{q} + k^2 \in \textup{Res}(H_\om^l)\cap B(0,r_0)^*\} \qquad  \qquad \qquad \qquad \qquad \qquad \qquad  \qquad 
\ee
$$ = \big\lbrace z_{q}(k) = \Lambda_{q} + k^2; \;  \textup{ such that } \varepsilon(l) i k \in {\mathcal Z}\big( B(0,r_0)^*,  A_{q}^{l}(\varepsilon(l) \cdot) \big)  \big\rbrace \setminus  {\mathcal F}_q,$$
where $B(0,r_0)^*= \{k \in \C ; \; 0<|k|\leq r_0 \}$ and $A_{q}^{l}$, $l=\infty, \gamma$, is defined by \eqref{defAql}.


Then, provided that 
\begin{equation}\label{cd1}
\textup{$I - \varepsilon(l) (A_{q}^{l})'(0)\Pi_{q}$ is an invertible operator}, 
\end{equation}
\big(with $\Pi_{q}$  the orthogonal projection onto $\text{ker}A_{q}^{l}(0)$\big),  \textbf{(ii)} and \textbf{(iii)} of Theorem \ref{theo1}  are immediate consequences of \textbf{(i)} and \textbf{(ii)} of Proposition \ref{P:P1} with $z = \varepsilon(l)ik$, because $A_{q}^{l}(0)$ is non negative.

Let us prove \textbf{(i)} of Theorem \ref{theo1} and interpret \eqref{cd1}. In order to apply  \textbf{(iii)} of Proposition \ref{P:P1}, we analyse the counting function of the eigenvalues of $A_{q}^{l}(0)$. 
According to \eqref{A0T}, 
$n\big( r, A_q^l(0) \big) = n \big( r, T_{q}^{l} \big)$, $T_{q}^{l}=\varepsilon(l) \Lambda_q^2 \, p_{q}W^{l}p_{q}$.
This together with Proposition \ref{propT1}, by using the mini-max principle, implies that for $l=\infty, \gamma$,
\begin{equation}\label{eqP5}
n \big(Cr,p_q \one_{K^\perp_0} p_q\big)
\leq n \big(r, A_q^l(0) \big) \leq 
n \big(r/C, p_q \one_{K^\perp_1} p_q \big).
\end{equation}
Since $K_{i}$ for $i = 0, 1$ are compact sets (with nonempty interior), then according to  \cite[Lemma 3.5]{RaWa02_01} we have
$n \big(r , p_q \one_{K^\perp_i} p_q \big)=\frac{\vert \ln r \vert}{\ln \vert \ln r \vert} \big( 1 + o(1) \big)$ as $r \searrow 0$. Combining this with \eqref{eqP5}, we deduce
\begin{equation}\label{eqP6}
n \big(r, A_q^l(0) \big) = \frac{\vert \ln r \vert}{\ln \vert \ln r \vert} \big( 1 + o(1) \big), \qquad r \searrow 0.
\end{equation}
We conclude the proof of  \textbf{(i)} of Theorem \ref{theo1} from \eqref{etape1}, \textbf{(iii)} of Proposition \ref{P:P1}, 
 \eqref{eqP6} and the following proposition giving the interpretation of the technical assumption \eqref{cd1} in terms of resonances.

\subsection{Interpretation of the assumption \eqref{cd1}}\label{interpretation}

\quad Using the notations of Subsection \ref{subP2}, let us introduce the operator 
\begin{equation}\label{defPq0}
P_q^0:= p_q \otimes r(0): e^{- \epsilon \langle x_3 \rangle} L^2(\R^3) \longrightarrow e^{\epsilon \langle x_3 \rangle}L^2(\R^3), \quad \epsilon >0,
\end{equation}
whose integral kernel is $\frac12{\mathcal P}_{q,b}(x_\perp,x_\perp^\prime)$ (see \eqref{ker1}).

%

\begin{proposition}\label{Prop:interpretation}
Let $A_q^l$ be the holomorphic operator-valued function defined by \eqref{defAql} and $\Pi_q$ the orthogonal projection onto Ker$A_q^l(0)$. Then the following assertions are equivalent:

$\textup{\textbf{(i)}}$ $I - \varepsilon(l) (A^l_q)^\prime(0) \Pi_{q}$ is invertible

$\textup{\textbf{(ii)}}$ The following limit exists for $ z $ in a sector $S_\delta := \big\lbrace z \in \C; \; \im (z) > \delta | \Re (z) -\Lambda_q| \big\rbrace $, $\delta >0$:
$$\lim_{S_\delta \ni z \longrightarrow \Lambda_q}    M_l (I+ik P^0_q) 
{\Tilde R}^l(z)
 M^l,$$
where $k=\sqrt{z - \Lambda_q}$, $\Im (k)>0$, $\Re (k) >0$, and ${\Tilde R}^l(z)$ is defined by \eqref{defRNR}.
\end{proposition}

Let us recall that $\textup{\textbf{(i)}}$ is the assumption  \eqref{cd1}, and before to prove the above result let us give an interpretation of $\textup{\textbf{(ii)}}$. From 
\eqref{dec1} and \eqref{dec2}, $z=\Lambda_q$ is an essential singularity of $M_l  ({H_0}-z)^{-1} M^l$ given by
$$
M_l ({H_0}-z)^{-1} M^l = -\frac{1}{ik}  M_l P^0_q M^l + \hbox{Hol}_{\Lambda_q}(z),
$$
where $\hbox{Hol}_{\Lambda_q}$ is a holomorphic operator valued function near $z=\Lambda_q$, given by 
$$
\hbox{Hol}_{\Lambda_q}(z) =
M_l (I+ ik P^0_q) ({H_0}-z)^{-1} M^l.
$$
For the above formula, we have used that in ${\mathcal L}\Big(e^{- \epsilon \langle x_3 \rangle} L^2(\R^3), e^{\epsilon \langle x_3 \rangle}L^2(\R^3)\Big)$
$$M_l P^0_q ({H_0}- ( \Lambda_q+k^2)  )^{-1} M^l= M_l P^0_q (\Lambda_q- ( \Lambda_q+k^2 ) )^{-1} M^l= -\frac{1}{k^2}  M_l P^0_q M^l. $$
Under obstacle perturbation, our main result shows  that $z=\Lambda_q$ remains an essential singularity. But it is not excluded that $\Lambda_q$ becomes also an isolated singularity coming from the perturbation of the holomorphic part $\hbox{Hol}_{\Lambda_q}(z)$. 

Our assumption \eqref{cd1} which is equivalent to $\textup{\textbf{(ii)}}$ does not allow this possibility. It is reasonable to think that \eqref{cd1} is  generic, for example in the sense that if $z=\Lambda_q$ becomes a isolated singularity of $M_l (I+ik P^0_q)  {\Tilde R}^l(z) M^l$ then, under a small perturbation of the obstacle $K$, this singularity disappears. In particular, for $l=\gamma$, among the possible singularities of $M_l (I+ik P^0_q)  {\Tilde R}^l(z) M^l$ there are the eigenvalues of the interior operator $H^{-\gamma}_K$ which has a discrete spectrum. Although it seems to be an open question,  we hope that if $\Lambda_q$ is an isolated eigenvalue of $H^{-\gamma}_K$, then under a small perturbation of $K$, this eigenvalue moves to another value close to  (but different from) $\Lambda_q$.

In order to simplify the statement of Theorem \ref{theo1}, let us introduced the following definition.
\begin{definition}\label{defnreso}
We will say that the obstacle $K$ doesn't produce an isolated resonance at $\Lambda_q$ if the property $\textup{\textbf{(ii)}}$ of Proposition \ref{Prop:interpretation} is satisfied.
\end{definition}

\begin{lemma}\label{equivA}
Near $k=0$, $\Im (k) >0$, $\Re (k) >0$, there exists a holomorphic  operator-valued function $K$ such that:
$$\left( I - \varepsilon(l) \frac{A_q^l(ik)}{ik} \right) = \Big( I - K(k) \Big) \, \Big( I - \varepsilon(l) \frac{\Lambda_q (\Lambda_q + k^2)}{ik}  M_l P^0_q M^l \Big)$$
$$\lim_{\substack{k \rightarrow 0 \\ \Re (k) > \delta \, \Im (k) >0} }  \! \! (I - K(k)) = I - \varepsilon(l) (A_{q}^{l})'(0)\Pi_{q},$$
where $\delta >0$ is fixed.


\end{lemma}
\begin{proof}
The proof of this Lemma follows similarly to that of Proposition 3.6 (or of Lemma 4.1) of \cite{BoBrRa14_01} where the same assumption \eqref{cd1} appears. We have 
$$
\frac{A_q^l(ik)}{ik} = \frac{A_q^l(0)(\Lambda_q + k^2)}{ik\Lambda_q} + \frac{A_q^l(ik)-\frac{\Lambda_q + k^2}{\Lambda_q}A_q^l(0) }{ik}
= \frac{A_q^l(0)(\Lambda_q + k^2)}{ik\Lambda_q} + (A_{q}^{l})'(0) + k R_2^l(ik),
$$
where $R_2$ is a holomorphic  operator valued function near $k=0$. Then, since $A_q^l(0)$ is self-adjoint, for $ik \in \C\setminus \R$, we have:
\bel{eqA1}
\left( I - \varepsilon(l) \frac{A_q^l(ik)}{ik} \right) = \Big( I - K(k) \Big) \, \left( I - \varepsilon(l) \frac{A_q^l(0)(\Lambda_q + k^2)}{ik\Lambda_q} \right),
\ee
with 
$$K(k) = \varepsilon(l)  (A_{q}^{l})'(0) \left(  I - \varepsilon(l) \frac{A_q^l(0)(\Lambda_q + k^2)}{ik\Lambda_q} \right)^{-1} + 
 \varepsilon(l)k R_2^l(ik) \left(  I - \varepsilon(l) \frac{A_q^l(0)(\Lambda_q + k^2)}{ik\Lambda_q} \right)^{-1}.$$

For $\Re (k) > \delta \, \Im (k) >0$, $|k|$ sufficiently small and $\nu(k)>0$ such that $\nu(k)= o(1)$, $|k|=o(\nu(k))$ as $|k|$ tends to $0$, we have:
\bel{borne1}
\|  \Big(  I - \varepsilon(l) \frac{A_q^l(0)(\Lambda_q + k^2)}{ik\Lambda_q} \Big)^{-1} \| = \sup_{\lambda_j \in \sigma(\varepsilon(l)A_q^l(0))} \!\frac{\Lambda_q |k|}{| ik\Lambda_q - \lambda_j(\Lambda_q + k^2)|} \leq \frac{ |k|}{| \Re (k) |} \leq \sqrt{1 + \frac{1}{\delta^2}},
\ee
\bel{borne2}
\| \one_{[\nu(k) , + \infty [}(A_q^l(0)) \Big(  I - \varepsilon(l) \frac{A_q^l(0)(\Lambda_q + k^2)}{ik\Lambda_q} \Big)^{-1} \| \leq \frac{ |k|}{\nu(k) - |\im  k |} \leq \frac{ |k|/\nu(k)}{1 - | k |/\nu(k)},
\ee
\bel{borne3}
 s\!\!-\!\!\!\lim_{|k| \rightarrow 0}  \one_{]0, \nu(k)[}(A_q^l(0)) = 0.
\ee
Then, combining the compactness of $(A_{q}^{l})'(0)$ with \eqref{borne3} and \eqref{borne1},  we obtain:
$$\lim_{\substack{k \rightarrow 0 \\ \Re (k) > \delta \, \Im (k) >0} }  \! \! (A_{q}^{l})'(0)  \one_{]0, \nu(k)[}(A_q^l(0))  \Big(  I - \varepsilon(l) \frac{A_q^l(0)(\Lambda_q + k^2)}{ik\Lambda_q} \Big)^{-1} =0,$$
and from  \eqref{borne2} and  \eqref{borne1}, we deduce
\bel{eqA2}
\lim_{\substack{k \rightarrow 0 \\ \Re (k) > \delta \,  \Im (k) >0} }  \! \! (I - K(k)) = 
 I - \varepsilon(l) (A_{q}^{l})'(0)\Pi_{q}.
\ee
This concludes the proof of Lemma \ref{equivA} by using the relations \eqref{eqA1}, \eqref{eqA2}, \eqref{opA_0} and \eqref{defPq0}.
\end{proof}

{\it Proof of Proposition \ref{Prop:interpretation}}
By definition of $A_q^l$, for $z \in S_\delta = \big\lbrace z \in \C; \; \im (z) > \delta | \Re (z)- \Lambda_q  | \big\rbrace $, $\delta >0$, and $k=\sqrt{z - \Lambda_q}$, $\Im (k)>0$, $\Re (k) >0$, we have
$$\left( I - \varepsilon(l) \frac{A_q^l(ik)}{ik} \right)= \left( I  + \varepsilon(l) M_l \Big(\frac{1}{z} - H_0^{-1}\Big)^{-1} M^l\right).$$
Then from \eqref{eqreso} and Lemma \ref{equivA}, we deduce, for $z= \Lambda_q + k^2 \in S_\delta$:
\bel{eqA4}
\Big( I - K(k) \Big) \, \Big( I - \varepsilon(l) \frac{\Lambda_q (\Lambda_q + k^2)}{ik}  M_l P^0_q M^l \Big) \,\, \Big( I  -  \varepsilon(l) M_l  \Big(\frac{1}{z} -{\Tilde R}^l(0)\Big)^{-1} M^l\Big) = I .
\ee
By exploiting that $M^lM_l= \varepsilon(l) V^l = \varepsilon(l) (H_0^{-1} -{\Tilde R}^l(0))$ \big(see \eqref{decompositionV}, \eqref{defVr}, \eqref{defVd}, \eqref{defRNR} \big), we have:
\bel{eqA5}
\left( I - \varepsilon(l) \frac{\Lambda_q (\Lambda_q + k^2)}{ik}  M_l P^0_q M^l \right) \,\, \left( I  -  \varepsilon(l) M_l  \left(\frac{1}{z} -{\Tilde R}^l(0)\right)^{-1} M^l\right) 
\ee
$$= I-\varepsilon(l) M_l  \left(\frac{1}{z} -{\Tilde R}^l(0)\right)^{-1} M^l
+ \varepsilon(l)   \frac{\Lambda_q (\Lambda_q + k^2)}{ik}  M_l P^0_q \left(H_0^{-1} - \frac1{z}\right) \left(\frac{1}{z} -{\Tilde R}^l(0)\right)^{-1}M^l.$$
Using that, in ${\mathcal L}\Big(e^{- \epsilon \langle x_3 \rangle} L^2(\R^3), e^{\epsilon \langle x_3 \rangle}L^2(\R^3)\Big)$,  $P^0_q H_0^{-1} = (\Lambda_q)^{-1} P^0_q$, \eqref{eqA5}  equals 
\bel{eqA6}
I -\varepsilon(l) M_l  (I + ik P^0_q) \Big(\frac{1}{z} -{\Tilde R}^l(0)\Big)^{-1}M^l= 
I -\varepsilon(l) M_l  (I + ik P^0_q) \Big(zI + z^2 {\Tilde R}^l(z)\Big) M^l
 \ee
 whose limit as $S_\delta \ni z \longrightarrow \Lambda_q$ exists if and only if 
\bel{eqA7}
\lim_{S_\delta \ni z \longrightarrow \Lambda_q}    M_l (I+ik P^0_q) {\Tilde R}^l(z) M^l
\ee
exists.

Then Proposition \ref{Prop:interpretation} follows from \eqref{eqA2}, \eqref{eqA4}, \eqref{eqA5}, \eqref{eqA6} and \eqref{eqA7}.

\begin{remark}
In Lemma \ref{equivA}, the limit is for $\arg (k) \in (0, \frac{\pi}{2}-\theta_\delta)$ with $\theta_\delta = \arctan \delta$, and then, in Proposition  \ref{Prop:interpretation}, it is sufficient to take $S_\delta = \big\lbrace z \in \C; \; \arg (z- \Lambda_q) \in (0, \pi-2\theta_\delta) \big\rbrace$.
\end{remark}

\section{Reduction to Toeplitz operators with symbol supported near the obstacle}\label{S:sec5}

\quad In this section we will prove Proposition \ref{propT1}. To the operators $V^l$, $l=\infty, \gamma$ defined on $L^2(\R^3)$ by \eqref{defVr}, \eqref{defVd} we associate the operator 
\bel{defWl}
W^l:= \frac{1}2 \int_{\R_{x_3}} V^l dx_3,
\ee
 defined on $L^2(\R^2)$ by \eqref{id4}

Our goal is to study the counting function of the compact non negative operators 
$$T_q^l:=   \Lambda_q^2\,  p_q W^l p_q, \qquad l=\infty, \gamma,$$
 where $p_q$ is the orthogonal projection onto ker$(H_{\text{Landau}}-\Lambda_q)$.

%

First, we study properties of $V^l$ in $L^2(\R^3)$.
For $q \in \N$, let us introduce a compact domain $K_1 \subset \R^3$ which contains $K$ and 
\bel{defEq}
\Ec_q (K_1)= \big\lbrace f \in L^2(\R^3) \cap C^\infty(\R^3); \; (H_0-\l_q)f= 0 \hbox{ on } K_1\big\rbrace.
\ee
It is an infinite dimensional subspace of $L^2(\R^3)$ which contains all functions $(P_q f^\perp \otimes \chi_3)$ when $ \chi_3 \in  L^2(\R_{x_3}) \cap C^\infty(\R_{x_3})$ satisfies $D^2_{x_3} \chi_3= 0$ on $I_{K_1}$, defined as for $K$ by \eqref{defIK}.

\begin{proposition}\label{prop42}
Fix $K_0$, $K_1$ two compact domains of $\R^3$, $K_0 \subset K \subset K_1$ with $\partial K_i \cap \partial K= \emptyset$, $i=0,1$. For $l=\infty, \gamma$, there exists $\Lc_q$ a finite codimension subspaces of $\Ec_q(K_1)$  and $C>1$ such that for any $f \in \Lc_q$,
\bel{equivV}
\frac1{C} \langle f ,  \one_{K_0} f \rangle_{L^2(\R^3)} \leq  \langle H_0 f, \varepsilon(l)V^l H_0 f \rangle_{L^2(\R^3)} 
\leq C \langle f ,  \one_{K_1} f \rangle_{L^2(\R^3)}  , \qquad l=\infty, \gamma.
\ee
\end{proposition}

\begin{proof}
The proof of the lower bound in the Dirichlet case is inspired by the analog result in the 2D case (see Proposition 3.1 of \cite{PuRo07}). By introducing the operator $H_0 + \one_{K_0}$, we have
$$V^\infty:= H_0^{-1} - (H_\om^\infty )^{-1} \oplus 0 = V^\infty_0 + V^\infty_1$$ with 
$$V^\infty_0 = H_0^{-1} - (H_0 + \one_{K_0})^{-1} , \qquad V^\infty_1=  (H_0 + \one_{K_0})^{-1}  - (H_\om^\infty )^{-1} \oplus 0.$$
Since the quadratic form associated to $(H_0 + \one_{K_0})$ coincide with $Q^\infty_\om$ \big(see \eqref{q3}\big) on $C_0^\infty(\om)$ \big(identified with $\big\lbrace u \in C_0^\infty(\R^3) {\mbox { supported in }} \om \big\rbrace$\big), then $V^\infty_1$ is a non negative operator on $L^2(\R^3)$.

Moreover, for $V^\infty_0 $, we have:
$$V^\infty_0 =H_0^{-1} \one_{K_0}  \Big( I - \one_{K_0}   (H_0 + \one_{K_0})^{-1}    \one_{K_0}  \Big)  \one_{K_0}H_0^{-1}.$$
Then exploiting that $\one_{K_0}   (H_0 + \one_{K_0})^{-1}    \one_{K_0} $ is a compact operator in ${\mathcal L}(L^2(\R^3))$, we deduce that, on a finite codimension subspace of $L^2(\R^3)$, we have:
$$ H_0 \,  V^\infty_0 \, H_0 \geq \frac12 \one_{K_0} .$$
This implies the lower bound of \eqref{equivV} in the Dirichlet case. 
The other estimates (lower bound for $l=\gamma$ and upper bounds) are  consequences of Lemma \ref{lemWT} and Lemma \ref{lemTSTo} below. 
\end{proof}

\begin{lemma}\label{lemWT}
Fix $q\in \N$. For $ l=\infty, \gamma$ and $K \subset K_1$ there exists $T^l_\Sigma$ an elliptic pseudo differential operator, on $L^2(\Sigma)$, of order $1$ and $\Lc_q$ a finite codimension subspaces of $\Ec_q(K_1)$  such that for any $f \in \Lc_q$,
$$
 \langle H_0 f, V^\infty H_0 f \rangle_{L^2(\R^3)} = \Lambda_q \langle f_{\vert K}, f_{\vert K}\rangle_{L^2(K)} -\langle f_{\vert \Sigma},T^\infty_\Sigma f_{\vert \Sigma}\rangle_{L^2(\Sigma)}.$$
$$
 \langle H_0 f, V^\gamma H_0 f \rangle_{L^2(\R^3)} =  -\langle f_{\vert \Sigma},T^\gamma_\Sigma f_{\vert \Sigma}\rangle_{L^2(\Sigma)}.$$
\end{lemma}

The above lemma is  comparable to Lemma 4.2 of  \cite{GoKaPe13}. The proof, which is closely related to the 2D case (see Subsection \ref{S:thm5.2} below), exploits the expressions of  $V^l$ in terms of {\it Dirichlet-Neumann} and {\it Robin-Dirichlet} operators \big(see \eqref{eq801} and \eqref{eq811}\big) and their elliptic properties as pseudo differential operators on $\Sigma$ (see Proposition \ref{defDRRD}). Moreover, for $f$ satisfying  $((\nabla^A)^2 +\Lambda_q)f = 0$ in $K$, there exists an elliptic pseudo differential operators $R^\gamma_q$ of order $1$ on $L^2(\Sigma)$
 such that  $ \dbdR^{A,\gamma} f = R^\gamma_q( f_{\vert \Sigma })  + F^\gamma_q( f)$ with $F^\gamma_q$ a finite rank operator (see Lemma \ref{lemRq}). In particular, for $\gamma =0$, $ {\dbdN} f = R^0_q( f_{\vert \Sigma })  + F^0_q( f)$ \big(for more details, we also refer to Remark 3.12 of \cite{GoKaPe13}\big).

As in the proof of Lemma 3.14 of  \cite{GoKaPe13} which doesn't depends on the even dimension of the space (see the end of Section 4 of \cite{GoKaPe13}), we have

\begin{lemma}\label{lemTSTo}
Fix $q\in \N$ and $K_i$, $i=0,1$ two compact domains of $\R^3$, $K_0 \subset K \subset K_1$, $\partial K_i \cap \partial K= \emptyset$. Let $T_\Sigma$ be a non negative  elliptic pseudo differential operator, on $L^2(\Sigma)$, of order $1$. 
Then there exists $\M_q$ a finite codimension subspaces of $\Ec_q(K_1)$  and $C>1$ such that for any $f \in \M_q$,
\bel{equivTS}
\frac1{C} \langle f ,  \one_{K_0} f \rangle_{L^2(\R^3)} \leq   \langle f_{\vert \Sigma}, T_\Sigma f_{\vert \Sigma}\rangle_{L^2(\Sigma)}
\leq C \langle f ,  \one_{K_1} f \rangle_{L^2(\R^3)} .
\ee
\end{lemma}

\begin{proof} {\it of Proposition \ref{propT1}.}
By definition of $W^l$ \big(see \eqref{defWl}\big) and exploiting the proof of  Lemma \ref{lemVM}, for $f^\perp \in ker (H_{\text{Landau}}-\Lambda_q)$ and any $\chi_3 \in C^\infty_c(\R_{x_3})$ equal to $1$ on $I_K$ \big(defined by \eqref{defIK}\big), we have
$$ \langle f^\perp , T_q^l f^\perp \rangle_{L^2(\R^2)} = \varepsilon(l) \frac{\Lambda_q^2}2 \int_{\R^3} f^\perp(x_1,x_2) \overline{(H_0^{-1} \chi_3 H_0 V^l H_0 \chi_3H_0^{-1} )(f^\perp \otimes 1_{\R})} (x_1,x_2,x_3) d x.
$$
For $f^\perp \in ker(H_{\text{Landau}}-\Lambda_q)$,  in $\Lcal \Big(e^{\epsilon <x_3>} L^2(\R^3), L^2(\R^3)\Big)$
$$  \chi_3H_0^{-1} (f^\perp\otimes 1_{\R})= \chi_3 (D^2_{x_3} + \Lambda_q)^{-1} (f^\perp\otimes 1_{\R})= \frac1{\l_q} (f^\perp\otimes \chi_3 ).$$
Then 
$$ \langle f^\perp , T_q^l f^\perp \rangle_{L^2(\R^2)}=  \frac{\varepsilon(l)}2 \langle H_0(f^\perp\otimes \chi_3), V^l H_0(f^\perp\otimes \chi_3)\rangle_{L^2(\R^3)},$$
and according to Proposition \ref{prop42}, for $K_0 \subset K \subset K_1$, such that $\partial K_i \cap \partial K= \emptyset$, $i=0,1$, there exists $C>1$ and $\Lc_q$ a finite codimension subspaces of $\Ec_q(K_1)$  such that for any $f^\perp \in \Lc_q$,
$$
\frac1{C} \langle (f^\perp\otimes \chi_3 ) ,  \one_{K_0} (f^\perp\otimes \chi_3 ) \rangle_{L^2(\R^3)} \leq   \langle f^\perp , T_q^l f^\perp \rangle_{L^2(\R^2)} 
\leq C \langle (f^\perp\otimes \chi_3 ) ,  \one_{K_1} (f^\perp\otimes \chi_3 ) \rangle_{L^2(\R^3)} .
$$
Let us choose $K_i$,  $i=0,1$, of the form $K_i= K^\perp_i \times I_i\subset \R^2 \times \R$ such that for $\chi_3 \in C_c^\infty(\R)$ equals to $1$ on $I_1 (\supset  I_0)$, we have:
$$\langle (f^\perp\otimes \chi_3 ) ,  \one_{K_i} (f^\perp\otimes \chi_3 ) \rangle_{L^2(\R^3)} = 
\int_{K_i} \vert f^\perp(x_1,x_2) \chi_3 (x_3) \vert^2  \one_{K^\perp_i}(x_1,x_2)  \one_{I_i}(x_3) d x_1d x_2d x_3$$
$$= \vert I_i \vert  \, \langle f^\perp ,  \one_{K_i^\perp} f^\perp \rangle_{L^2(\R^2)} .$$
This implies \eqref{equivT} because $f^\perp= p_q f^\perp$ .
\end{proof}

\section{Boundary operators}\label{S:sec6}

\quad In this section we recall how the method of layer potential allows to prove that the {\it Dirichlet-Neumann} et {\it Neumann-Dirichlet} operators are pseudo differential operators on a surface and how Lemma \ref{lemWT} and Lemma \ref{lemTSTo} follow. In presence of a constant magnetic field, these technics was already used in \cite{Pe09}, \cite{GoKaPe13}  for even-dimensional cases.

\subsection{Green kernel for $(\nabla^A)^2$ near the diagonal}

\quad Let $G_0(x,y)$, $x,y \in \R^3$ be the integral kernel of $H_0^{-1}$. It is related to ${\mathcal H}_0(t,x,y)$, the heat kernel, by the formula:
$$ G_0(x,y) = \int_0^{+ \infty} {\mathcal H}_0(t,x,y) dt,$$
where \big(see e.g. \cite{AvHeSi78_01}\big), for $x=(x_1,x_2,x_3)=(x_\perp,x_3) \in \R^2\times \R$, 
$${\mathcal H}_0(t,x,y)= \frac{1}{\sqrt{4 \pi t}} \frac{b}{4 \pi \sinh (bt)} \exp \left\lbrace - \frac{(x_3-y_3)^2}{4t} - \frac{b}4 \coth (bt) \vert x_\perp - y_\perp \vert^2  - i \frac{b}2 x_\perp  \wedge   y_\perp \right\rbrace,$$
with $\vert x_\perp \vert^2= x_1^2 + x_2^2$, $x_\perp  \wedge   y_\perp = x_1 y_2 - x_2 y_1$.
Then we obtain:

\begin{lemma}\label{lemKernel}
The integral kernel of $H_0^{-1}$ is smooth outside the diagonal and we have
\begin{equation}\label{eqK1}
G_0(x,y) = K_0(x,y) + \mathcal{O}(1) \quad \hbox{as} \quad \vert x-y \vert \rightarrow 0,
\end{equation}
with
\begin{equation}\label{eqK2}
K_0(x,y) \sim \frac{e^{-\frac{ib}{2}x_\perp \wedge y_\perp}}{4 \pi \vert x-y \vert} 
\sim \frac{1}{4 \pi \vert x-y \vert} \quad \hbox{as} \quad \vert x-y\vert \rightarrow 0.
\end{equation}
Moreover, for  $x, y \in \Sigma$, $ (\dbdN)_y G_0(x,y)$ satisfies the corresponding behavior as $\vert x-y\vert$ tends to $0$, where $(\dbdN)_y$ means that the differentiation is with respect to the variable $y$. More precisely, we have 
\begin{equation}\label{eqK3}
(\dbdN)_y G_0(x,y) = (\dbdN)_y \left( \frac{e^{-\frac{ib}{2}x_\perp \wedge y_\perp}}{4 \pi \vert x-y \vert} \right) 
+ \mathcal{O}(1) \quad \hbox{as} \quad \vert x-y\vert \rightarrow 0.
\end{equation}
\end{lemma}

\begin{proof}

By the change of variables $u = bt$, we can rewrite $G_0(x,y)$ as
\begin{equation}\label{eqK4}
\small{G_0(x,y) = e^{-\frac{ib}{2}x_\perp \wedge y_\perp} I(x,y), \qquad
I(x,y) := \frac{b^\frac{1}{2}}{(4 \pi)^{\frac{3}{2}}} \int_0^{+ \infty} \frac{e^{ -\frac{b}{4} \left[ \frac{(x_3-y_3)^2}{u} + \coth (u) \vert x_\perp - y_\perp \vert^2 \right]}}{u^\frac{1}{2} \sinh (u)} du}.
\end{equation}
Then Lemma \ref{lemKernel} is a direct consequence of the following lemma.
\begin{lemma}\label{lemKerbis}
\textup{\textbf{(i)}} The function $I(x,y)$ defined by \eqref{eqK4} can be rewritten as
$$
I(x,y) = I_0(x,y) + I_\infty(x,y),
$$
where

\textup{\textbf{a)}} $I_\infty(x,y) = \mathcal{O}(1)$ uniformly with respect to the variables $x, y$.

\textup{\textbf{b)}} The function $I_0(x,y)$ satisfies for $\vert x - y \vert \ll 1$ 
\begin{equation}\label{eqK5}
\small{I_0(x,y) = \frac{1}{(4 \pi)^\frac{3}{2} \vert x - y \vert} \int_{b \vert x - y \vert^2}^{+\infty} \frac{e^{-\frac{u}{4}}}{u^\frac{1}{2}} du + \mathcal{O}(1).}
\end{equation}
\textup{\textbf{(ii)}} The function $(\dbdN)_y G_0(x,y)$ satisfies for $\vert x - y \vert \ll 1$
\begin{equation}\label{eqK6}
 \small{(\dbdN)_y G_0(x,y) = \small{-2 i \sum_{j=1}^3 \nu_j A_j G_0(x,y) +
 \frac{\nu \cdot (x-y) e^{-\frac{ib}{2}x_\perp \wedge y_\perp}}{2(4 \pi)^{\frac{3}{2}}\vert x - y \vert^3}} 
  \int_{b \vert x - y \vert^2}^{+\infty} u^\frac{1}{2} e^{-\frac{u}{4}} du
}+ \mathcal{O}(1).
\end{equation}
\end{lemma}

The proof of Lemma \ref{lemKerbis} is of computational nature. Hence, for more transparency in the presentation, it is differed in the Appendix. Now let us back to the proof of Lemma \ref{lemKernel}.

Identities \eqref{eqK1} and \eqref{eqK2} follows immediately from \textup{\textbf{(i)}} of Lemma \ref{lemKerbis} together with \eqref{eqK4} and remarking that $\int_0^{+\infty} u^{-\frac{1}{2}} e^{-\frac{u}{4}} du = 4 \int_0^{+\infty} e^{-v^2} du = (4\pi)^{\frac{1}{2}}$.

Identity \eqref{eqK3} follows from  \textup{\textbf{(ii)}} of Lemma \ref{lemKerbis} remarking firstly that 
$\frac{1}{2 }\int_0^{+\infty} u^{\frac{1}{2}} e^{-\frac{u}{4}} du = (4 \pi)^{\frac{1}{2}}$, and secondly that
$$
(\dbdN)_y \left( \frac{e^{-\frac{ib}{2}x_\perp \wedge y_\perp}}{4\pi \vert x-y \vert} \right)
= \frac{e^{-\frac{ib}{2}x_\perp \wedge y_\perp} \nu \cdot (x - y)}{4 \pi \vert x-y \vert^3}
-2 i\sum_{j=1}^3 \nu_j A_j G_0(x,y).
$$
This concludes the proof of Lemma \ref{lemKernel}.
\end{proof}

\subsection{Boundary operators associated to $(\nabla^A)^2$}\label{ss62}
{
\quad According to the properties of $G_0$ near the diagonal, the following single-layer and double-layer potentials of a function $f$ on $\Sigma$ are well defined:
\bel{defScal}
\Sc f(x) := \int_\Sigma f(y) G_0(x,y) d\sigma(y), \qquad x \in \R^3 \setminus \Sigma,
\ee 
\bel{defDcal}
\Dc f(x) := \int_\Sigma f(y) (\dbdN)_y G_0(x,y) d\sigma(y), \qquad x \in \R^3\setminus \Sigma
\ee 
\big(see for instance \cite{Tay96}\big). Moreover, for $x \in \Sigma$ we have the following limit relations:
\bel{eqS}
\lim_{z \rightarrow x} \Sc f_\om (z) = \Sbf f_\Sigma(x)  = \lim_{z \rightarrow x} \Sc f_K (z),
\ee
\bel{eqDext}
\lim_{z \rightarrow x} \Dc f_\om (z) = -  \frac12 f_\Sigma(x) + \Dbf f_\Sigma, 
\ee
 \bel{eqDint}
\lim_{z \rightarrow x} \Dc f_K (z) = \frac12 f_\Sigma(x) + \Dbf f_\Sigma, 
\ee
where 
\bel{defSbf}
\Sbf f_\Sigma (x) := \int_\Sigma f(y) G_0(x,y) d\sigma(y), \qquad x \in \Sigma,
\ee 
\bel{defDbf}
\Dbf f_\Sigma (x) := \int_\Sigma f(y) (\dbdN)_y G_0(x,y) d\sigma(y), \qquad x \in \Sigma,
\ee 
define compact operators on $L^2(\Sigma)$. More precisely, following the arguments  of Section 7.11 of  \cite{Tay96} \big(see also Lemmas 3.2, 3.3 and 3.6 of \cite{GoKaPe13}\big), $\Sbf $ and $\Dbf $ are pseudo differential operators, on $\Sigma$, of order $(-1)$, and $\Sbf $ is an elliptic self-adjoint operator on $L^2(\Sigma)$ which is an isomorphism from $L^2(\Sigma)$ onto $H^1(\Sigma)$. Moreover for $\varphi \in C^\infty(\Sigma)$ and $\bullet = K, \om$, $f_\bullet:= \Sc (\Sbf^{-1}\varphi ) _{\vert \bullet}$ is the unique solution of 
\bel{eq1}
 \left\{
 \begin{array}{ccc}
 (\nabla^A)^2 f_\bullet = 0 & \hbox{ in } & \bullet\\
 f_{\bullet \vert \Sigma }= \varphi,
 \end{array}
 \right.
 \ee
 and we have:
 \bel{eqBd1}
 \Sbf ( \dbdN f_K) = (\Dbf - \frac12) \varphi , \qquad \Sbf ( \dbdN f_\om) = (\Dbf + \frac12) \varphi .
 \ee

Inserting $\dbdR = \dbdN + \gamma$ above, we obtain:
\bel{eqBd2}
 \Sbf ( \dbdR f_K) = (\Dbf + \Sbf \gamma - \frac12) \varphi , \qquad \Sbf ( \dbdR f_\om) = (\Dbf + \Sbf \gamma + \frac12) \varphi .
 \ee

\begin{remark}
Due to the ellipticity of $\Sbf$ and the compactness of $\Dbf$ and $\Sbf$, the operators $(\Dbf +\Sbf \gamma  \pm \frac12)$ are Fredholm operators  and consequently there inverse exist on finite codimension spaces. In other words, there exists elliptic pseudo differential operators $R_{\pm}$ of order $0$,  such that $R_{\pm} \, (\Dbf +\Sbf \gamma  \pm \frac12)-I_{L^2(\Sigma)}$  and $(\Dbf +\Sbf \gamma  \pm \frac12) \, R_{\pm} -I_{L^2(\Sigma)}$ are finite rank operators.

Moreover, as in \cite{GoKaPe13} (see Lemma 3.7 and Corollary 3.10), for all $\varepsilon \in [-1,1] $ outside a finite subset \big($-\varepsilon \notin \sigma(\Sbf^{-1} (\Dbf \pm \frac12) + \gamma)  $\big), the operators $(\Dbf + \Sbf (\gamma + \varepsilon) \pm \frac12)$ are invertible on $L^2(\Sigma)$.
\end{remark}

On this way, as in the 2D case \big(see Proposition 3.8 of \cite{GoKaPe13}\big) we can give the definition of the {\it Dirichlet-Robin} and {\it Robin-Dirichlet} operators introduced in \eqref{eq801} and \eqref{eq811} :
\begin{proposition}\label{defDRRD}

\textup{\textbf{(i)}} The interior (resp. exterior) Dirichlet-Robin operator
${\bf DR_K}$ (resp. ${\bf DR_\om}$) defined by 
\begin{equation}\label{opDN}
{\bf DR_K} = \Sbf^{-1} \left(\Dbf + \Sbf \gamma - \frac12 \right) ,  \qquad  {\bf DR_\Omega} = \Sbf^{-1} \left(\Dbf + \Sbf \gamma + \frac12\right),
\end{equation}
are first order elliptic pseudo differential operators on $\Sigma$. The  interior (resp. exterior) Dirichlet-Neumann operator
${\bf DN_K}$ (resp. ${\bf DN_\om}$) corresponds to $\gamma=0$.

\textup{\textbf{(ii)}} The interior (resp. exterior) Robin-Dirichlet operator
${\bf RD_K}$ (resp. ${\bf RD_\om}$) defined on a finite codimensional subspace of $L^2(\Sigma)$ by 
\begin{equation}\label{opND}
{\bf RD_K} = \left(\Dbf + \Sbf \gamma - \frac12 \right)^{-1}\Sbf  ,  \qquad  {\bf RD_\Omega} = \left(\Dbf + \Sbf \gamma + \frac12\right)^{-1} \Sbf,
\end{equation}
are elliptic pseudo differential operators on $\Sigma$ of order $(-1)$. The  interior (resp. exterior) Neumann-Dirichlet operator
${\bf ND_K}$ (resp. ${\bf ND_\om}$) corresponds to $\gamma=0$.
\end{proposition}

\subsection{Proof of Lemma \ref{lemWT}}\label{S:thm5.2}

From \eqref{eq801} and \eqref{eq811}, for $f \in \Ec_q (K_1)$ (defined by \eqref{defEq}), we have:
$$
\langle H_0 f , V^\infty H_0 f \rangle_{L^2(\R^3)}  = 
\langle  f_{|K} , \Lambda_q f_{| K} \rangle_{L^2(K)} + \langle f_{|\Sigma} \, , \, \partial_N^{A}f \rangle_{L^2(\Sigma)}-\langle f_{|\Sigma} \, , \,{\bf DN_\Omega} ( f_{|\Sigma}) \rangle_{L^2(\Sigma)}  ,
$$
and
$$
\langle H_0 f, V^\gamma H_0 f \rangle_{L^2(\R^3)}= - \langle  \partial_\Sigma^{A,\gamma} f \, , \,( {\bf RD_K}-{\bf RD_\Omega})  \partial_\Sigma^{A,\gamma} f \rangle_{L^2(\Sigma)}.
$$

Moreover, based on the methods of Section \ref{ss62}, by considering the Green function associated to the operator $(\nabla^A)^2+\Lambda_q$ on $K$, we construct $\Sbf_q $ and $\Dbf_q$ two pseudo differential operators, on $\Sigma$, of order $(-1)$, with $\Sbf_q $  elliptic  on $L^2(\Sigma)$ but not necessarly invertible.
These operators satisfy relations like \eqref{eqBd1}, \eqref{eqBd2}. Then following the proofs of Lemma 3.11 and Remark 3.12 of \cite{GoKaPe13}, we obtain:

  \begin{lemma}\label{lemRq}
  There exists an elliptic pseudo differential operators $R^\gamma_q$ of order $1$ on $L^2(\Sigma)$
  and a finite rank  
  operator $F^\gamma_q$ 
  such that for $f$ satisfying  $((\nabla^A)^2 +\Lambda_q)f = 0$ in $K$,  
 \bel{eq11}
 \partial_\Sigma^{A,\gamma} f := \partial_N^A f + \gamma f=  R^\gamma_q( f_{\vert \Sigma })
 + F^\gamma_q( f).
 \ee
 \end{lemma}

 Then we deduce Lemma \ref{lemWT} from Proposition \ref{defDRRD} and Lemma \ref{lemRq} with 
 $$T^\infty_\Sigma = {\bf DN_\Omega} - R^0_q, \qquad T^\gamma_\Sigma=(R^\gamma_q)^* ( {\bf RD_K}-{\bf RD_\Omega})R^\gamma_q.$$

}

\section{Appendix}\label{s7}

\quad This appendix is devoted to the proof of Lemma \ref{lemKerbis}. Constants in the $\mathcal{O}(\cdot)$ are generic, namely changing from a relation to another.
\\

\textbf{(i)} Let $I(x,y)$ be the function defined by \eqref{eqK4}. Define $I_0(x,y)$ and $I_\infty(x,y)$ by
\begin{equation}\label{eqa1}
\small{I_0(x,y) := \frac{b^\frac{1}{2}}{(4 \pi)^{\frac{3}{2}}} \int_0^{1} \frac{e^{ -\frac{b}{4} \left[ \frac{(x_3-y_3)^2}{u} + \coth (u) \vert x_\perp - y_\perp \vert^2 \right]}}{u^\frac{1}{2} \sinh (u)} du}
\end{equation}
and 
\begin{equation}\label{eqa2}
\small{I_\infty(x,y) := \frac{b^\frac{1}{2}}{(4 \pi)^{\frac{3}{2}}} \int_1^{+\infty} \frac{e^{ -\frac{b}{4} \left[ \frac{(x_3-y_3)^2}{u} + \coth (u) \vert x_\perp - y_\perp \vert^2 \right]}}{u^\frac{1}{2} \sinh (u)} du.}
\end{equation}
Since $\coth(u) \geq 1$ for $u \geq 1$, then clearly $I_\infty(x,y) = \mathcal{O}(1)$ uniformly with respect to $x$, $y$. This gives  \textbf{(i)} \textbf{a)} of Lemma \ref{lemKerbis}. 

Now let us prove \textbf{(i)} \textbf{b)}. By using the change of variables $u = 1/v$, the integral $I_0(x,y)$ given by \eqref{eqa1} verifies
\begin{equation}\label{eqa3}
\frac{(4\pi)^\frac{3}{2}}{b^{\frac{1}{2}}} \small{I_0(x,y) = \int_1^{+\infty} \frac{e^{ -\frac{b \vert x - y \vert^2 v}{4}}}{v^\frac{1}{2}} \frac{e^{ -\frac{b \vert x_\perp - y_\perp \vert^2}{4} \left[ \coth (\frac{1}{v}) - v  \right]}}{v \sinh (\frac{1}{v})} dv.}
\end{equation}
It can be easily checked that for $\vert x - y \vert \ll 1$, and uniformly with respect to $v \in [1,+ \infty[$, we have
\begin{equation}\label{eqa6}
\small{\frac{e^{-\frac{b \vert x_\perp - y_\perp \vert^2}{4} \left[ \coth (\frac{1}{v}) - v  \right]}}{v \sinh (\frac{1}{v})}
 = 1 + \mathcal{O} \left( \frac{1}{v} \right), \qquad v \geq 1.}
\end{equation} 
By combining \eqref{eqa3} and \eqref{eqa6}, we get for $\vert x - y \vert \ll 1$
\begin{equation}\label{eqa7}
\small{\frac{(4\pi)^\frac{3}{2}}{b^{\frac{1}{2}}} I_0(x,y) = \int_1^{+\infty} \frac{e^{ -\frac{b \vert x - y \vert^2 v}{4}}}{v^\frac{1}{2}}dv + \mathcal{O}(1).}
\end{equation}
Now \textbf{(i)} \textbf{b)} of Lemma \ref{lemKerbis} is a direct consequence of \eqref{eqa7} using the change of variables $u = b \vert x - y \vert^2 v$.
\\

\textbf{(ii)} The proof of this point is quite similar to that of the previous. Let $G_0(x,y)$ and $I(x,y)$ be the functions defined by \eqref{eqK4}. By a direct computation, it can be checked that 
\begin{equation}\label{eqa8}
\small{(\dbdN)_y G_0(x,y) = e^{-\frac{ib}{2}x_\perp \wedge y_\perp} \sum_{j=1}^3 \nu_j \partial_j I(x,y) - 2 i\sum_{j=1}^3 \nu_j A_j G_0(x,y),}
\end{equation}
where the differentiation is with respect to the variable $y$. So to conclude, it suffices to investigate the integral functions $\partial_j I(x,y) = \partial_j \big( I_0(x,y) + I_\infty(x,y) \big)$, $j = 1,2,3$, where $I_0(x,y)$ and $I_\infty(x,y)$ are the functions defined respectively by \eqref{eqa1} and \eqref{eqa2}. Firstly, an easy computation show that we have
\begin{equation}\label{eqa9}
\small{\partial_jI_\infty(x,y)} = (x_j - y_j) \mathcal{O}(1), \qquad j = 1, 2, 3.
\end{equation}
Secondly, by using for example the expression \eqref{eqa3} of $I_0(x,y)$, it can be checked that
\begin{equation}\label{eqa10}
\small{\frac{(4\pi)^\frac{3}{2}}{b^{\frac{1}{2}}} \partial_j I_0(x,y) = \frac{b(x_j - y_j)}{2}\int_1^{+\infty} v^\frac{1}{2} e^{ -\frac{b \vert x - y \vert^2 v}{4}} \frac{\coth (\frac{1}{v}) e^{ -\frac{b \vert x_\perp - y_\perp \vert^2}{4} \left[ \coth (\frac{1}{v}) - v  \right]}}{v^2 \sinh (\frac{1}{v})} dv, \quad j = 1, 2,}
\end{equation}
and 
\begin{equation}\label{eqa11}
\small{\frac{(4\pi)^\frac{3}{2}}{b^{\frac{1}{2}}} \partial_3 I_0(x,y) = \frac{b(x_3 - y_3)}{2}\int_1^{+\infty} v^\frac{1}{2} e^{ -\frac{b \vert x - y \vert^2 v}{4}} \frac{e^{ -\frac{b \vert x_\perp - y_\perp \vert^2}{4} \left[ \coth (\frac{1}{v}) - v  \right]}}{v \sinh (\frac{1}{v})} dv.}
\end{equation}
Similarly to the expansion \eqref{eqa6}, it can be proved that the functions 
$h(v) := \frac{\coth (\frac{1}{v}) e^{ -\frac{b \vert x_\perp - y_\perp \vert^2}{4} \left[ \coth (\frac{1}{v}) - v  \right]}}{v^2 \sinh (\frac{1}{v})}$ and 
$k(v) := \frac{e^{ -\frac{b \vert x_\perp - y_\perp \vert^2}{4} \left[ \coth (\frac{1}{v}) - v  \right]}}{v \sinh (\frac{1}{v})}$
appearing respectively in the integrals \eqref{eqa10} and \eqref{eqa11} satisfy for $\vert x - y \vert \ll 1$
\begin{equation}
\small{h(v) = 1 + \mathcal{O} \left( \frac{1}{v} \right), \quad k(v) = 1 + \mathcal{O} \left( \frac{1}{v} \right), \qquad v \geq 1.}
\end{equation}
This together with \eqref{eqa10} and \eqref{eqa11} give for $j = 1, 2, 3$ and $\vert x - y \vert \ll 1$
\begin{equation}
\small{\frac{(4\pi)^\frac{3}{2}}{b^{\frac{1}{2}}} \partial_j I_0(x,y) = \frac{b(x_j - y_j)}{2} \left[ \int_1^{+\infty} v^\frac{1}{2} e^{ -\frac{b \vert x - y \vert^2 v}{4}} dv +  \int_1^{+\infty} \frac{e^{ -\frac{b \vert x - y \vert^2 v}{4}}}{v^\frac{1}{2}} dv \mathcal{O}(1) \right].}
\end{equation}
After the change of variables $u = b \vert x - y \vert^2 v$, finally we get
\begin{equation}\label{eqa12}
\small{\frac{(4\pi)^\frac{3}{2}}{b^{\frac{1}{2}}} \partial_j I_0(x,y) = \frac{b(x_j - y_j)}{2} \left[ \frac{1}{b^\frac{3}{2} \vert x - y \vert^3} \int_{b \vert x - y \vert^2}^{+\infty} u^\frac{1}{2} e^{-\frac{u}{4}} du + \frac{1}{b^\frac{1}{2} \vert x - y \vert} \int_{b \vert x - y \vert^2}^{+\infty} \frac{e^{-\frac{u}{4}}}{u^\frac{1}{2}} du \mathcal{O}(1) \right].}
\end{equation}
Consequently, \eqref{eqK6} of point \textbf{(ii)} follows from \eqref{eqa8}, \eqref{eqa9} and \eqref{eqa12}. This concludes the proof of the Lemma \ref{lemKerbis}.

{\sl Acknowledgments.}

The authors are grateful to G. Raikov for his continued support and helpful exhange of views.
The first author thanks the Mittag-Leffler Institut where this work was initiate with useful discussions with M. Persson and with G. Rozenblum. 

  V. Bruneau was partially supported by ANR-08-BLAN-0228. 
D. Sambou is partially supported by the Chilean 
Program \textit{N\'ucleo Milenio de F\'isica Matem\'atica
RC$120002$}.

\bibliographystyle{amsplain}
\providecommand{\bysame}{\leavevmode\hbox to3em{\hrulefill}\thinspace}
\providecommand{\MR}{\relax\ifhmode\unskip\space\fi MR }
\providecommand{\MRhref}[2]{%
  \href{http://www.ams.org/mathscinet-getitem?mr=#1}{#2}
}
\providecommand{\href}[2]{#2}

\end{document}